\documentclass[a4paper]{article}
\usepackage{graphicx}
\usepackage{amsfonts}
\usepackage{amsmath}
\usepackage{color}
\usepackage{bibmods}
\usepackage{accents}
\usepackage{hyperref}

\newtheorem{theorem}{Theorem}

\newtheorem{definition}[theorem]{Definition}

\newtheorem{lemma}[theorem]{Lemma}

\newtheorem{proposition}[theorem]{Proposition}

\newenvironment{proof}[1][Proof]{\textit{#1.} }{\ \rule{0.5em}{0.5em}}

\begin{document}

\title{Asymptotic description of solutions of the exterior Navier Stokes
problem in a half space}
\author{Matthieu Hillairet \\
{\small  CEREMADE}\\
{\small Universit\'{e} Paris Dauphine, France}\\
{\small hillairet@ceremade.dauphine.fr} \and Peter Wittwer\thanks{%
Supported in part by the Swiss National Science Foundation.} \\
{\small D\'{e}partement de Physique Th\'{e}orique}\\
{\small Universit\'{e} de Gen\`{e}ve, Switzerland}\\
{\small peter.wittwer@unige.ch}}
\date{\today }
\maketitle

\begin{abstract}
We consider the problem of a body moving within an incompressible fluid at
constant speed parallel to a wall, in an otherwise unbounded domain. This
situation is modeled by the incompressible Navier-Stokes equations in an
exterior domain in a half space, with appropriate boundary conditions on the
wall, the body, and at infinity. 
We focus on the case where the size of the body is small. We prove in a very general setup that  the solution of this problem
is unique and we compute a sharp decay rate of the solution far from the moving body and the 
wall.
\end{abstract}

\tableofcontents

\section{Introduction}

In the present paper we discuss solutions of the Navier-Stokes equations for
the stationary flow around a body that moves with constant speed parallel to
a wall in an otherwise unbounded space filled with a fluid. The mathematical
formulation of the problem is as follows. Let $\Omega _{+}=\{\mathbf{x}%
=(x,y)\in \mathbb{R}^{2}\left\vert {}\right. y>1\}$, and let $B_{t}=\{{(x,y)}%
\in \mathbb{R}^{2}\left\vert {}\right. {(x,y)}+t\mathbf{e}_{1}\in B\}$,
where $\mathbf{e}_{1}=(0,1)$ 
and where $B$ is a bounded open connected subset of $\Omega _{+}$ such that $%
\overline{B}\subset \Omega _{+}$. 
As a function of $t\geq 0$, the set $B_{t}$ corresponds to a body which is
immersed in a fluid and moves at constant speed from right to left parallel
to the wall $\partial \Omega _{+}$. The flow around this body is modeled by
the Navier-Stokes equations 
\begin{equation}
\left\{ 
\begin{array}{rcl}
\partial _{t}\mathbf{U} & = & -\left( \mathbf{U}\cdot \nabla \right) \mathbf{%
U}+\Delta \mathbf{U}-\nabla P~, \\ 
\nabla \cdot \mathbf{U} & = & 0~,%
\end{array}%
\right.  \label{EU}
\end{equation}%
in $\Omega _{t}=$ $\Omega _{+}\setminus \overline{B_{t}}$ with the boundary conditions
(the boundary $\partial \Omega _{+}$ is at rest and we choose no slip
boundary conditions at the surface of the body), 
\begin{equation}
\left. \mathbf{U}\right\vert _{\partial \Omega _{+}}=0~,\quad \lim\limits 
_{\substack{ {\mathbf{x}\,\in \,\Omega _{t}}  \\ {|\mathbf{x}|}\rightarrow
\infty }}\mathbf{U}({\mathbf{x}},t)=0~,\quad \left. \mathbf{U}\right\vert
_{\partial B_{t}}=-\mathbf{e}_{1}~.  \label{BU}
\end{equation}%
We are interested in the construction of solutions of equations (\ref{EU})-%
(\ref{BU}) that are stationary when viewed in a reference frame attached to
the moving body. We therefore set 
\begin{equation*}
\mathbf{U}({\mathbf{x}},t)=\mathbf{u}(\mathbf{x}+t\mathbf{e}_{1})~,\quad P(%
\mathbf{x},t)=p(\mathbf{x}+t\mathbf{e}_{1})~,
\end{equation*}%
and get the following stationary problem:%
\begin{equation}
\left\{ 
\begin{array}{rcl}
-\left( \mathbf{u}\cdot \nabla \right) \mathbf{u}-\partial _{x}\mathbf{u}%
+\Delta \mathbf{u}-\nabla p & = & 0~, \\ 
\nabla \cdot \mathbf{u} & = & 0~,%
\end{array}%
\right.  \label{Problem21}
\end{equation}%
in $\Omega _{+}\setminus \overline{B}$, with the boundary conditions%
\begin{equation}
\left. \mathbf{u}\right\vert _{\partial B}=-\mathbf{e}_{1}~,\quad \left. 
\mathbf{u}\right\vert _{\partial \Omega _{+}}=0~,\quad \lim\limits 
_{\substack{ \mathbf{x}\,\in \,\Omega _{+}  \\ |\mathbf{x}|\rightarrow
\infty }}\mathbf{u}(\mathbf{x})=0~.  \label{BC21}
\end{equation}%

\noindent Note that we have set without restriction of generality all the
physical constants and the speed of the moving body equal to one. This can
always be achieved by an appropriate scaling. With this choice of
normalization the Reynolds number of the moving body corresponds to the
diameter $\varepsilon $ of $B$. The problem contains a second length-scale,
which is the distance $h$ of (the center of) $B$ from the wall $\partial
\Omega _{+}$. In this paper, we are interested in the regime where $%
\varepsilon $ is small, 
and in particular small
with respect to $h$.

The system \eqref{Problem21} with boundary conditions \eqref{BC21} is 
related to
the so-called exterior Navier-Stokes problem: 
\begin{equation}
\left\{ 
\begin{array}{rcl}
-\lambda \left( (\mathbf{u}-\mathbf{u}_{\infty })\cdot \nabla \right) 
\mathbf{u}+\Delta \mathbf{u}-\nabla p & = & 0~, \\ 
\nabla \cdot \mathbf{u} & = & 0~,%
\end{array}%
\right. \quad \text{ in $\mathbb{R}^{n}\setminus \overline{B}$}
\label{ExteriorNS}
\end{equation}%
\begin{equation}
\left. \mathbf{u}\right\vert _{\partial B}=\mathbf{u}^{\ast }~,\quad
\lim\limits_{|\mathbf{x}|\rightarrow \infty }\mathbf{u}(\mathbf{x})=0~.
\label{BCE}
\end{equation}%
where $B$ is a bounded open connected subset of $\mathbb{R}^{n}$ with smooth
boundary, $\lambda \in \mathbb{R}$ is the Reynolds number, $\mathbf{u}%
_{\infty }\in \mathbb{R}^{n}$ is a prescribed asymptotic velocity and $%
\mathbf{u}^{\ast }\in H^{1/2}(\partial B)$ is a given boundary condition.
Most of the methods for solving this problem are extensively described in
the fundamental book of G.P. Galdi \cite{Galdi98}. We give a brief 
outline of some results in the following lines.

\medskip

The first methods  to solve such exterior problems 
go back to the pioneering work of J. Leray \cite{Leray33}. In this reference,
the author introduces an \emph{invading domain} method yielding existence of
at least one \emph{weak solution} to \eqref{ExteriorNS}-\eqref{BCE} whose
velocity-field $\mathbf{u}$ satisfies $\Vert \nabla \mathbf{u} \: ; \: L^{2}(\mathbb{%
R}^{n}\setminus \overline{B})\Vert <\infty .$ A comparable result is
obtained by H. Fujita \cite{Fujita61}. Similar weak solutions are
constructed also for exterior Navier Stokes system \eqref{ExteriorNS} with
other types of boundary conditions on $\partial B$ (see \cite{Serre87} and 
\cite{Weinberger73}). The 
only shortcoming of these weak solutions is that insufficient information is
obtained on the behavior at infinity. In the case $n=2$ with $\mathbf{u}%
_{\infty }=0,$ it is still not known whether the vanishing condition at infinity is satisfied by weak solutions or not
(see \cite{Amick91,Gilbarg&Weinberger78} and \cite{Russo10} for recent
developments in this question). This difficulty is linked to the famous Stokes Paradox which holds in two space-dimensions. 
For the geometry of the present paper, existence of weak solutions for \eqref{Problem21} decaying at infinity, combined with other boundary conditions, 
is studied in \cite{Hillairet07}.

\medskip

In the case $\mathbf{u}_{\infty }\neq 0,$ a more refined description of the
asymptotic behavior of solutions to \eqref{ExteriorNS}-\eqref{BCE} is given in a
second series of papers. These results rely on the idea that the dominating
system at infinity is the Oseen system : 
\begin{equation}
\left\{ 
\begin{array}{rcl}
\lambda \mathbf{u}_{\infty }\cdot \nabla \mathbf{u}+\Delta \mathbf{u}-\nabla
p & = & 0~, \\ 
\nabla \cdot \mathbf{u} & = & 0~.%
\end{array}%
\right.  \label{Oseen}
\end{equation}%
A detailed comprehension of the asymptotics of solutions to this linear
system enables to construct solution to \eqref{ExteriorNS}-\eqref{BCE} via a
standard perturbation technique and then to compute the asymptotics of the
constructed solutions. Such an analysis is performed by K.I. Babenko in the
3D-setting \cite{Babenko75}, and by R. Finn and D.R. Smith \cite%
{Finn&Smith67} and L.I. Sazonov \cite{Sazonov03} in the 2D-setting. This
method is transposed to the geometry studied in the present paper by T.
Fischer, G.C. Hsiao and W.L. Wendland in \cite{Fischer&Hsiao&Wendland86}. 
In this last case  the difficulty linked to the Stokes paradox is less limitative. 
In particular, the Stokes problem : 
\begin{equation}
\left\{ 
\begin{array}{rcl}
\Delta \mathbf{u}-\nabla p & = & 0~, \\ 
\nabla \cdot \mathbf{u} & = & 0~,%
\end{array}%
\right. \quad \text{in $\Omega _{+}\setminus \overline{B}.$}
\end{equation}%
\begin{equation}
\left. \mathbf{u}\right\vert _{\partial B}=\mathbf{u}^{\ast }~,\quad
\left. \mathbf{u}\right\vert _{\partial \Omega _{+}}=0~,\quad 
\lim\limits_{|\mathbf{x}|\rightarrow \infty }\mathbf{u}(\mathbf{x})=0~,
\label{BCS}
\end{equation}%
is well-posed. In \cite{Fischer&Hsiao&Wendland86} existence of solutions to %
\eqref{Problem21}-\eqref{BC21} is obtained via a perturbation method based on this
linear Stokes problem. Nevertheless, the dominating system at
infinity in our case is still the Oseen system with $\mathbf{u}_{\infty }=%
\mathbf{e}_{1}$ so that no precise asymptotics of the constructed solutions
is given in \cite{Fischer&Hsiao&Wendland86}.  
This computation requires a very careful analysis of the Oseen linear system in the half space,
the analysis of which is not yet available with these former methods.
For completness, we mention here that the properties of the Stokes system in the geometry of the present paper is
studied in the more general framework of weighted Sobolev spaces in \cite%
{Amrouche&Bonzom09}. No equivalent study for the Oseen system is provided to
our knowledge.

\medskip

The present paper 
uses a dynamical-system approach for studying the asymptotics of solutions to 
an exterior Navier Stokes problem. In this 
method, the first idea is to interpret one coordinate as a time. Then, one
rewrites \eqref{Problem21} as a system of nonlinear evolution equations.
Solutions are constructed via a perturbation method in function spaces
enabling to compute the exact long-time behavior. In return, one obtains
solutions to \eqref{Problem21}-\eqref{BC21} with detailed asymptotics. This
program is applied successfully 
to the case of the 3D exterior Navier-Stokes system in \cite{Wittwer06} and of
the 2D half-space problem, 
with the solid $B$ replaced by a smooth source term with compact support, in a
previous publication of the authors \cite{Hillairet&Wittwer09}. In this last
reference, the solutions to the system of nonlinear evolution equations are
computed performing a Fourier transform in the transversal direction (\emph{%
i.e.}, with respect to $x$ in our case). This is the reason why we replace
the obstacle by a source term with compact support in \cite%
{Hillairet&Wittwer09}.

\medskip

In the present paper, we prove existence of solutions to \eqref{Problem21}-\eqref%
{BC21} with a detailed asymptotics  by combining the invading method 
of Leray and the dynamical-system approach. 
Since we apply  in part perturbation methods, our 
results hold only for small Reynolds numbers. 
The role of Reynolds number is played by the diameter of the solid $B$ in our setting. More precisely, 
let $S$ be a bounded open subset of $\mathbb{R}^{2}$ containing the origin, with a smooth boundary, 
and let $h$ be a positive parameter which fixes the center of the body with respect to the boundary. Then,
we set $S_{\varepsilon }:=(0,1+h)+\varepsilon S$ and rewrite our system 
as : 
\begin{equation}
\left\{ 
\begin{array}{rcl}
-\left( \mathbf{u}\cdot \nabla \right) \mathbf{u}-\partial _{x}\mathbf{u}%
+\Delta \mathbf{u}-\nabla p & = & 0~, \\ 
\nabla \cdot \mathbf{u} & = & 0~,%
\end{array}%
\right.  \label{Problem2}
\end{equation}%
in $\Omega \setminus \overline{S}_{\varepsilon }$, with the boundary
conditions%
\begin{equation}
\left. \mathbf{u}\right\vert _{\partial S_{\varepsilon }}=-\mathbf{e}%
_{1}~,\quad \left. \mathbf{u}\right\vert _{\partial \Omega _{+}}=0~,\quad
\lim\limits_{\substack{ \mathbf{x}\,\in \,\Omega _{+}  \\ |\mathbf{x}%
|\rightarrow \infty }}\mathbf{u}(\mathbf{x})=0~.  \label{BC2}
\end{equation}%
In what follows \eqref{Problem2} together with boundary conditions \eqref{BC2}
is referred to as {Problem 1}. The following theorem is our main result. 

\begin{theorem}
\label{maintheorem} For $\varepsilon$ sufficiently small, there exists a
unique weak solution $\mathbf{u}$ of Problem$~$1. Furthermore, there exists
a constant $C_{\varepsilon }<\infty $ such that, for all $(x,y)\in \Omega_+
\setminus \overline{S_{\varepsilon}}$,%
\begin{equation}
|\mathbf{u}(x,y)|\leq \dfrac{C_{\varepsilon }}{y^{\frac{3}{2}}}~.
\label{decay}
\end{equation}
\end{theorem}

A precise definition of weak solutions for Problem 1 is given in Section 1.
For the sake of simplicity we only give a bound for the decay of the weak
solution in \eqref{decay}. Nevertheless, a precise first order for the
asymptotics is available with our techniques. Such computations are
performed in an independent paper (see \cite{Christoph}). This bound is
the critical ingredient for proving uniqueness in the frame of weak
solutions to Problem 1.

\medskip

Our strategy to obtain detailed information on weak solutions of Problem~1
at infinity is divided in five steps. First, we show the existence of weak
solutions for Problem$~$1 by the invading method of Leray. Second, we use a
cut-off function to obtain, from a weak solution $(\mathbf{u},p)$ of Problem$%
~$1, a weak solution $(\mathbf{\tilde{u}},\tilde{p})$ to 
\begin{equation}
\begin{array}{c}
-\left( \mathbf{\tilde{u}}\cdot \nabla \right) \mathbf{\tilde{u}}-\partial
_{x}\mathbf{\tilde{u}}+\Delta \mathbf{\tilde{u}}-\nabla \tilde{p}=\mathbf{f}%
~, \\ 
\nabla \cdot \mathbf{\tilde{u}}=0~,%
\end{array}
\label{Problem1}
\end{equation}%
with the boundary conditions 
\begin{equation}
\left. \mathbf{\tilde{u}}\right\vert _{\partial \Omega _{+}}=0~,\quad
\lim\limits_{\substack{ \mathbf{x}\in \Omega _{+}  \\ \mathbf{x}\rightarrow
\infty }}\mathbf{\tilde{u}}(x,y)=0~,  \label{BC1}
\end{equation}%
The system \eqref{Problem1} together with boundary conditions \eqref{BC1} is
referred to as Problem 2 in what follows. Note that, in this new system we
keep the divergence-free condition: the cut-off function is applied to the
stream function of $\mathbf{u}.$ This enables to compute explicitly the
source term $\mathbf{f}.$ So, in the third step, we show that, for $%
\varepsilon $ small enough, the function $\mathbf{f}$ satisfies the
smallness condition formulated in our previous paper \cite%
{Hillairet&Wittwer09}, so that there exists at least one $\alpha$-solution $%
(\mathbf{u}_{\alpha },p_{\alpha })$ for Problem$~$2 (see Section 3 for the
definition of $\alpha $-solutions). In the fourth step, we prove a
weak-strong uniqueness result for Problem$~$2. Once again, this weak-strong
argument applies to weak-solutions and $\alpha $-solutions constructed for $%
\varepsilon $ small enough. The uniqueness of solutions for Problem$~$2 does
not directly imply the uniqueness of solutions for Problem$~$1, because
different solutions of Problem$~$1 may lead to different functions $\mathbf{f%
}$. So in a last step, we prove uniqueness of weak solutions for Problem$~$1
for $\varepsilon $ small enough.

\subsection{Sets and function spaces}

\label{sec_functionspaces} In the whole paper, we use the standard notations
for function spaces such as $L^{p}(\mathcal{O})$ for Lebesgue spaces and $%
W^{m,p}(\mathcal{O})$ or $H^{m}(\mathcal{O})$ for Sobolev spaces. We denote
by $\mathcal{C}^{m}(\mathcal{O})$ the spaces of continuous functions having $%
m$ continuous derivative ($m$ might be infinite). We use the subscript $c$
to specify that function have compact support in the set $\mathcal{O}.$
Given a Banach space $X$ and $p\in X$, the norm of $p$ in $X$ is denoted by $%
\Vert p;X\Vert $ and,  { if $X$ is a function space containing the constants:}  
\begin{equation}
\Vert p;{X/\mathbb{R}}\Vert :=\inf \{\Vert p+c;X\Vert ,~c\in \mathbb{R}\}~.
\label{normmod}
\end{equation}%
This latter notation is very useful for pressures which are defined up to an
additive constant in systems such as \eqref{Problem21}.

\medskip

In some proofs, we shall need a smooth covering of $\Omega _{+}$ or of $%
\Omega _{+}\setminus  \{(0,1+h)\}$. For this purpose, we introduce here some particular subsets of $\Omega
_{+}.$ First, we denote ${B(}\lambda )$ the open balls with center $(0,1+h)$ \emph{%
i.e.} given $\lambda >0,$ we denote by 
\begin{equation*}
B(\lambda )=\Big\{(x,y)\in \Omega _{+}\text{ such that }{}%
|(x,y)-(0,1+h)|<\lambda \Big\}~.
\end{equation*}%
We note in particular, that, since $S$ is bounded, there exists $\varepsilon _{0}>0$%
, such that $S_{\varepsilon }\subset {B(}h/3)$ for $\varepsilon <\varepsilon
_{0}.$ We keep the classical convention $B((x,y),r)$ for balls with center $(x,y) \in \mathbb R^2$
and radius $r >0.$
We introduce $(\Delta _{n})_{n\in \mathbb{N}}$ an increasing covering
of $\Omega _{+}$ such that, for all $n\in \mathbb{N}$ :

\begin{itemize}
\item $\Delta_n$ has a smooth boundary

\item $B(\,(1+n)h\,)\subset \Delta _{n}\subset B(\,(2+n)h\,){.}$
\end{itemize}
Furthermore we define, for $n\in \mathbb{N}$, the sets $\mathcal{A}_{n}$ by 
$\mathcal{A}_{n}=\Delta _{n}\setminus \overline{{B(}2^{-n}h)}.$ Therefore, for all $n\in \mathbb{N},$ $\mathcal{A}_{n}$ has a smooth
boundary.

\section{\label{Section 2}Weak solutions for Problem$~$1}

In this section, we consider the theory of weak solutions for Problem$~$1. %
The main result of this section is the following theorem. 

{
\begin{theorem}
\label{thm_cvg0main} There exists a family $(\mathbf{u}_{\varepsilon})_{\varepsilon>0}$  which is defined for $%
\varepsilon$ sufficiently small and such that:

\begin{enumerate}
\item[(i)] for all $\varepsilon>0,$ $\mathbf{u}_{\varepsilon}$ is a weak solution of Problem 1 for $S_{\varepsilon},$

\item[(ii)] given $\eta >0$ there exists $0<\varepsilon _{\eta }$ such that,
for all $\varepsilon <\varepsilon _{\eta }$ there holds $\Vert \mathbf{u}%
_{\varepsilon };D\Vert \leq \eta $,

\item[(iii)] there exists a pressure $p_{\varepsilon }$ such that $(\mathbf{u}%
_{\varepsilon },p_{\varepsilon })$ satisfies \eqref{Problem1} in $\Omega
_{+}\setminus \overline{S}_{\varepsilon }$ and, given $m\in \mathbb{N}$,
there holds: 
\begin{equation}
\Vert \mathbf{u}_{\varepsilon };{\mathcal{C}^{m+1}(}\overline{{ \mathcal{A}%
_{2}}}{)}\Vert +\Vert p_{\varepsilon };{\mathcal{C}^{m}(}\overline{{ \mathcal{%
A}_{2}}}{)/\mathbb{R}}\Vert \leq C_{m}\Vert \mathbf{u}_{\varepsilon };D\Vert
~.
\end{equation}%
for some universal constant $C_{m}$ depending only on $m$.
\end{enumerate}
\end{theorem}
}

We refer the reader to the introduction for the definition of { $\mathcal{A}%
_{2}$}. We introduce function spaces and the definition of weak solutions for
Problem~1 just below.

\medskip

{The proof of this result is divided in three steps. First, we recall the
method of Leray for the construction of weak solutions. We obtain in this
way a family of weak solutions which satisfy a particular uniform bound with
respect to the (small) size of the obstacle. Eventually, we prove that this
family of solutions tends to $0$ in the sense of }{\bf Theorem \ref{thm_cvg0main}}.

\subsection{Definition of weak solutions}

{To begin with, the size $\varepsilon $ of the obstacle is fixed such that $%
{S}_{\varepsilon }\subset B(h/3).$ } Let $(\mathbf{u},p)$ be a
smooth solution of Problem$~$1 for $S_{\varepsilon }$. We extend $\mathbf{u}$
{from} $\Omega _{+}\setminus \overline{S_{\varepsilon }}$ to the whole of $%
\Omega _{+}$ by setting $\mathbf{u}=-\mathbf{e}_{1}$ on $\overline{%
S_{\varepsilon }}$. Let $\mathbf{w}$ be a smooth divergence-free
vector-field with compact support in $\Omega _{+}$ which is equal to a given
constant vector field $\mathbf{W}$ on $S_{\varepsilon }$. Then, if we
multiply equation (\ref{Problem2}) by $\mathbf{w}$ and integrate over $%
\Omega _{+}\setminus \overline{S_{\varepsilon }}$ we get 
\begin{equation}
\int_{\Omega _{+}\setminus \overline{S_{\varepsilon }}}(\Delta \mathbf{u}%
-\nabla p)\cdot \mathbf{w}~d\mathbf{x}=\int_{\Omega _{+}\setminus \overline{%
S_{\varepsilon }}}\left[ (\mathbf{u}+\mathbf{e}_{1})\cdot \nabla \mathbf{u}%
\right] \cdot \mathbf{w}~d\mathbf{x}~.  \label{w1}
\end{equation}%
In order to unburden notations we have suppressed in (\ref{w1}), and in
what follows, the arguments of functions when no confusion is possible.
Applying Green's identity to the left-hand side of (\ref{w1}) leads to the
equality 
\begin{equation*}
\int_{\Omega _{+}\setminus \overline{S_{\varepsilon }}}(\Delta \mathbf{u}%
-\nabla p)\cdot \mathbf{w}~d\mathbf{x}=\int_{\partial (\Omega _{+}\setminus 
\overline{S_{\varepsilon }})}T(\mathbf{u},p)\mathbf{n}\cdot \mathbf{w}~d%
\text{$\sigma $}-\frac{1}{2}\int_{\Omega _{+}\setminus \overline{%
S_{\varepsilon }}}\left( \nabla \mathbf{u+}\left[ \nabla \mathbf{u}\right]
^{\top }\right) :\left( \nabla \mathbf{w+}\left[ \nabla \mathbf{w}\right]
^{\top }\right) ~d\mathbf{x}~,
\end{equation*}%
where $T(\mathbf{u},p)=(\nabla \mathbf{u}+[\nabla \mathbf{u}]^{\top })-pI$
and where $\mathbf{n}$ is the outward normal on $\partial (\Omega
_{+}\setminus \overline{S_{\varepsilon }})$. Using the boundary conditions
for $\mathbf{u}$, which imply in particular that $\nabla \mathbf{u}$
vanishes on $S_{\varepsilon }$, and using that $\mathbf{w}=\mathbf{W}$ on $%
S_{\varepsilon }$, we obtain that $\mathbf{u}$ satisfies%
\begin{equation}
\int_{\Omega _{+}}\nabla \mathbf{u}:\nabla \mathbf{w}~d\mathbf{x}%
+\int_{\Omega _{+}}\left[ (\mathbf{u}+\mathbf{e}_{1})\cdot \nabla \mathbf{u}%
\right] \cdot \mathbf{w}~d\mathbf{x}=-\mathbf{\Sigma }\cdot \mathbf{W}~,
\label{WF1}
\end{equation}%
with the vector 
\begin{equation}
\mathbf{\Sigma }=-\int_{\partial S_{\varepsilon }}T(\mathbf{u},p)\mathbf{n}~d%
\text{$\sigma $}~.  \label{Signa}
\end{equation}%
The vector $\mathbf{\Sigma }$ is the force which the fluid exerts on $%
S_{\varepsilon }$. If we replace, on a formal level, $\mathbf{w}$ by $%
\mathbf{u}$ in (\ref{WF1}), we obtain, as $\mathbf{W=-e}_{1}$ in this
case :
\begin{equation}
\int_{\Omega _{+}}|\nabla \mathbf{u}|^{2}~d\mathbf{x}+\int_{\Omega _{+}}%
\left[ (\mathbf{u}+\mathbf{e}_{1})\cdot \nabla \mathbf{u}\right] \cdot 
\mathbf{u}~d\mathbf{x}=\mathbf{\Sigma }\cdot \mathbf{e}_{1}~.  \label{12}
\end{equation}%
Integrating by parts yields, as $\mathbf u$ is divergence-free: %
\begin{equation}
\int_{\Omega _{+}}\left[ (\mathbf{u}+\mathbf{e}_{1})\cdot \nabla \mathbf{u}%
\right] \cdot \mathbf{u}~d\mathbf{x}=\dfrac{1}{2}\int_{\Omega _{+}}\left[ 
\mathbf{u}\cdot \nabla |\mathbf{u}|^{2}+\partial _{x}|\mathbf{u}|^{2}\right]
~d\mathbf{x}=0~,  \label{13}
\end{equation}%
and therefore (\ref{Signa}) reduces to 
\begin{equation}
\int_{\Omega _{+}}|\nabla \mathbf{u}|^{2}~d\mathbf{x}=\mathbf{\Sigma }\cdot 
\mathbf{e}_{1}~.  \label{NRJ1}
\end{equation}%
We conclude that if $(\mathbf{u},p)$ is a 
solution of 
Problem$~$1 which decays sufficiently rapidly at infinity, then $\mathbf{u}$
satisfies the integral equation (\ref{WF1}) and we have the identity (\ref%
{NRJ1}), which means that $\nabla \mathbf{u}$ $\in L^{2}(\Omega _{+})$.

\medskip

The above discussion motivates the following functional setting for weak
solutions of Problem$~$1. Let $\mathcal{D}$ be the vector space of smooth
divergence-free vector-fields with compact support in $\Omega _{+}$. We
equip $\mathcal{D}$ with the scalar product 
\begin{equation}
((\mathbf{w}_{1},\mathbf{w}_{2}))=\int_{\Omega _{+}}\nabla \mathbf{w}%
_{1}:\nabla \mathbf{w}_{2}~d\mathbf{x}~.  \label{sp}
\end{equation}%
For functions in $\mathcal{D}$ we have%
\begin{equation}
\int_{\Omega _{+}}\nabla \mathbf{w}_{1}:\nabla \mathbf{w}_{2}~d\mathbf{x}=%
\frac{1}{2}\int_{\Omega _{+}}(\nabla \mathbf{w}_{1}+[\nabla \mathbf{w}%
_{1}]^{\top }):(\nabla \mathbf{w}_{2}+[\nabla \mathbf{w}_{2}]^{\top })~d%
\mathbf{x}~.
\end{equation}%
{Let $D$ be the Hilbert space with respect to the scalar product (\ref{sp})
obtained by completion of $\mathcal{D}.$ Let $\mathcal{D}^{\varepsilon
}\subset \mathcal{D}$ be the vector-fields $\mathbf{w}\in \mathcal{D}$ which
are constant on $S_{\varepsilon }$ and let $D^{\varepsilon }$ be the closure
of $\mathcal{D}^{\varepsilon }$ in $D$.} On $D^{\varepsilon }$ we define the
function $\Gamma $ by 
\begin{equation}
\begin{array}{rccl}
\Gamma : & D^{\varepsilon } & \longrightarrow & \mathbb{R}^{2} \\ 
& \mathbf{w} & \longmapsto & \mathbf{W=}\dfrac{1}{\left\vert S_{\varepsilon
}\right\vert }\displaystyle{\int_{S_{\varepsilon }}}\mathbf{w(x)}~d\mathbf{x}%
~.%
\end{array}
\label{Gamma}
\end{equation}%
{It} follows from Hardy's inequality {(see }Proposition \ref{prop_Hardy}{\ 
below)} that $\Gamma $ is bounded. For convenience later on we define, for all $%
{\mathbf{W}}\in \mathbb{R}^{2}$,%
\begin{equation}
\mathcal{D}_{\mathbf{W}}^{\varepsilon }=\{\mathbf{w}\in \mathcal{D}%
^{\varepsilon }\left\vert {}\right. \mathbf{w}_{|_{S_{\varepsilon }}}={%
\mathbf{W}}\}~,\quad D_{\mathbf{W}}^{\varepsilon }=\{\mathbf{w}\in
D^{\varepsilon }\left\vert {}\right. \mathbf{w}_{|_{S_{\varepsilon }}}={%
\mathbf{W}}\}~.
\end{equation}%
{Such spaces have been studied extensively in  \cite[Chapter III.5]{Galdi98}. In
particular, we emphasize that with our smoothness assumptions on ${\partial S}_{\varepsilon }$ we have that}
$\overline{\mathcal{D}_{\mathbf{W}}^{\varepsilon }}=D_{\mathbf{W}%
}^{\varepsilon }$.

\bigskip

\noindent {Following the work of Leray, we }now define weak solutions for
Problem$~$1:

\begin{definition}
\label{defweaksolution}A vector-field $\mathbf{u}$ is called a weak solution
of Problem$~$1, if

\begin{enumerate}
\item[(i)] $\mathbf{u}\in {D}_{-\mathbf{e}_{1}}^{\varepsilon }$,

\item[(ii)] there exists a vector $\mathbf{\Sigma }\in \mathbb{R}^{2}$, such
that for all $\mathbf{w}\in {\mathcal{D}^{\varepsilon }}$%
\begin{equation}
\int_{\Omega _{+}}\nabla \mathbf{u}:\nabla \mathbf{w}~d\mathbf{x}%
+\int_{\Omega _{+}}\left[ (\mathbf{u}+\mathbf{e}_{1})\cdot \nabla \mathbf{u}%
\right] \cdot \mathbf{w}~d\mathbf{x}=-\mathbf{\Sigma }\cdot \Gamma (\mathbf{w%
})~,  \label{eq_wf}
\end{equation}%
and 
\begin{equation}
\int_{\Omega _{+}}|\nabla \mathbf{u}|^{2}~d\mathbf{x}\leq {\mathbf{\Sigma }}%
\cdot \mathbf{e}_{1}~.  \label{eq_NRJ}
\end{equation}
\end{enumerate}
\end{definition}

The following {standard} lemma shows that weak solutions are well defined.

\begin{lemma}
\label{lem_trilinear}Let $(\mathbf{u},\mathbf{v})\in D^{2}$ and let $\mathbf{%
w}\in D$ with $\mathrm{Supp}(\mathbf{w})\subset \mathcal{O}\subset \subset
\Omega _{+}\footnote{%
We use the standard notation $A\subset \subset B$ to mean that the closure $%
\bar{A}$ is a compact subset of $B$.}$. Then,%
\begin{equation}
\int_{\Omega _{+}}\left[ (\mathbf{u}+\mathbf{e}_{1})\cdot \nabla \mathbf{v}%
\right] \cdot \mathbf{w}~d\mathbf{x}=-\int_{\Omega _{+}}\left[ (\mathbf{u}+%
\mathbf{e}_{1})\cdot \nabla \mathbf{w}\right] \cdot \mathbf{v}~d\mathbf{x}~,
\label{eq_trilinearsym}
\end{equation}%
and%
\begin{equation}\left\vert \displaystyle{\int_{\Omega _{+}}}\left[ (\mathbf{u}+\mathbf{e}%
_{1})\cdot \nabla \mathbf{v}\right] \cdot \mathbf{w}~d\mathbf{x}\right\vert 
\leq C(\mathcal{O})\left( \Vert \mathbf{u};L^{4}(\mathcal{O})\Vert ~\Vert 
\mathbf{v};{D}\Vert ~\Vert \mathbf{w};L^{4}(\mathcal{O})\Vert +\Vert \mathbf{%
v};D\Vert ~\Vert \mathbf{w};L^{2}(\mathcal{O})\Vert \right) ~.
\label{eq_trilinear}
\end{equation}
\end{lemma}

\bigskip

Below we show that, given a weak solution $\mathbf{u}$ of Problem 1, one can
construct a function $p$ such that the couple $(\mathbf{u},p)$ satisfies the
equation (\ref{Problem2}) in the classical sense. We will call $p$ the
pressure associated with the weak solution $\mathbf{u}$. Using 
the ellipticity of the Stokes operator together with the smoothness of the
boundary of the fluid domain, it is possible to prove that $(\mathbf{u}%
,p)\in \mathcal{C}^{\infty }(\overline{\Omega _{+}}\setminus S_{\varepsilon
})$, and that the boundary conditions (\ref{BC2}) on $S_{\varepsilon }$ and
on $\partial \Omega _{+}$ are satisfied in the classical sense. Therefore,
weak solutions have all the requested properties of classical solutions, and
the only difficulty with weak solutions is that their rate of decay at
infinity remains unknown.
A bound on the decay rate, like (\ref{decay}), is crucial in order to prove
uniqueness of solutions.

\subsection{Existence of weak solutions}

{In this section, we prove: }

\begin{theorem}
\label{thm_existencews} There exist constants $K<\infty $ and $\varepsilon
_{1}>0$ such that if $\varepsilon <\varepsilon _{1}$, there exists at least
one weak solution $\mathbf{u}$ for Problem$~$1 for $S_{\varepsilon }$,
satisfying the further bound $\Vert \mathbf{u};D\Vert \leq K$.\smallskip 
\end{theorem}
The proof is based on the exhaustion method of Leray. Namely, we
consider a nested sequence of finite domains that converge to $\Omega_+$
and, for any domain of this sequence, we prove existence of one approximate
weak solution having support in this domain and satisfying a suitable
estimate. Our result then follows by a compactness argument. Many aspects of
the proof are standard, but the uniform bound is new to our knowledge.

\subsubsection{\label{sec_WS0}{Sketch of proof for Theorem }\protect\ref%
{thm_existencews}}

In this proof the size $\varepsilon $ of the obstacle is again fixed such
that ${ S}_{\varepsilon }\subset B(h/3)$%
. We mention further assumptions on $\varepsilon $ when needed. We consider
the sequence $\left( \Delta _{n}\right) _{n\geq 1}$ given in the
introduction. This sequence satisfies, for all $n\in \mathbb{N}$ :

\begin{itemize}
\item $\Delta _{n}$ is a bounded open  set having a smooth
boundary

\item $S_{\varepsilon }\subset \subset \Delta _{n}\subset \Delta _{n+1}$

\item $\bigcup_{n\in \mathbb{N}}\Delta _{n}=\Omega _{+}$.
\end{itemize}

Given $\Delta _{n}$, we define $D^{\varepsilon ,n}$ and $D_{\mathbf{W}%
}^{\varepsilon ,n}$ by%
\begin{equation}
D^{\varepsilon ,n}=\{\mathbf{w}\in D^{\varepsilon }\left\vert {}\right. 
\mathbf{w}_{|_{\Omega _{+}\setminus \overline{\Delta _{n}}}}=0\}~,\quad D_{%
\mathbf{W}}^{\varepsilon ,n}=\{\mathbf{w}\in D_{\mathbf{W}}^{\varepsilon
}\left\vert {}\right. \mathbf{w}_{|_{\Omega _{+}\setminus \overline{\Delta
_{n}}}}=0\}~.
\end{equation}

\bigskip

\noindent {With these conventions, the definition of approximate weak
solutions for Problem$~$1 }%
is:
\begin{definition}
\label{apf} {Let $n\in \mathbb{N}$. A vector-field $\mathbf{u}$ is called an
approximate weak solution on $\Delta _{n}$ if:}

\begin{enumerate}
\item[$(i)$] $\mathbf{u}\in {D}_{-\mathbf{e}_{1}}^{\varepsilon ,n}$,

\item[$(ii)$] for all $\mathbf{w}\in {D}_{0}^{\varepsilon ,n}$,%
\begin{equation}
\int_{\Omega _{+}}\nabla \mathbf{u}:\nabla \mathbf{w}~d\mathbf{x}%
+\int_{\Omega _{+}}\left[ (\mathbf{u}+\mathbf{e}_{1})\cdot \nabla \mathbf{u}%
\right] \cdot \mathbf{w}~d\mathbf{x}=0~.  \label{eq_wftrunc}
\end{equation}
\end{enumerate}
\end{definition}

{Before giving a sketch of the proof of }{\bf Theorem \ref{thm_existencews}}, {we
mention that, since }$D_{0}^{\varepsilon ,n}$ is a closed subspace of $%
D^{\varepsilon ,n}$ of codimension two, the Lagrange multiplier theorem
implies the existence of a vector $\mathbf{\Sigma }\in \mathbb{R}^{2}$, such
that for all $\mathbf{w}\in {D}^{\varepsilon ,n}$%
\begin{equation}
\int_{\Omega _{+}}\nabla \mathbf{u}:\nabla \mathbf{w}~d\mathbf{x}%
+\int_{\Omega _{+}}\left[ (\mathbf{u}+\mathbf{e}_{1})\cdot \nabla \mathbf{u}%
\right] \cdot \mathbf{w}~d\mathbf{x}=-\mathbf{\Sigma }\cdot \Gamma (\mathbf{w%
})~.  \label{eq_wftruncext}
\end{equation}%
The vector $\mathbf{\Sigma }$ is the force associated with the approximate
weak solution $\mathbf{u}$. Since $\mathbf{u}\in D^{\varepsilon ,n}$, we can
replace $\mathbf{w}$ by $\mathbf{u}$ in (\ref{eq_wftruncext}), and an
integration by parts yields%
\begin{equation}
\int_{\Omega _{+}}|\nabla \mathbf{u}|^{2}~d\mathbf{x}={\mathbf{\Sigma }}%
\cdot \mathbf{e}_{1}~.  \label{eq_NRJtrunc}
\end{equation}%
The energy (in)equality is therefore a consequence of Definition \ref{apf},
and for this reason we do not need to impose it in the definition of
approximate weak solutions in contrast with the definition of weak solutions.

\bigskip

\noindent The proof of {\bf Theorem \ref{thm_existencews}} is based on the
following two lemmas:

\begin{lemma}
\label{lemexistence}There exists a constant $\varepsilon _{1}>0$ such that {%
if} $\varepsilon <\varepsilon _{1}$ there exists at least one approximate
weak solution on $\Delta _{n}$, for
all $n\in \mathbb{N}$.
\end{lemma}

\begin{lemma}
\label{lemuniformestimate}Let $\varepsilon $ be as in Lemma \ref%
{lemexistence} and $n\in \mathbb{N}$. There exists a constant $K<\infty $,
such that $\Vert \mathbf{u};D\Vert +\left\vert \mathbf{\Sigma }\right\vert
\leq K$ for any approximate  weak solution $\mathbf{u}$ on $\Delta _{n}$ with associated force $\mathbf{\Sigma 
}$.
\end{lemma}

{Proofs for these lemmas are given in Section \ref{Lemma8}. We sketch now the
remaining steps of the proof of }{\bf Theorem \ref{thm_existencews}}
assuming that $\varepsilon <\varepsilon _{1}$.

\bigskip

\noindent $(i)$ {By Lemma \ref{lemexistence}, there exists} a sequence $(%
\mathbf{u}_{n},\mathbf{\Sigma }_{n})_{n \geq 1}$ such that $\mathbf{u}_{n}$ is an
approximate weak solution on $\Delta _{n}$ with associated force $\mathbf{%
\Sigma }_{n}$. By Lemma \ref{lemuniformestimate} this sequence is bounded in 
$D\times \mathbb{R}^{2}$. One can therefore extract a subsequence $(\mathbf{u%
}_{n_{i}},\mathbf{\Sigma }_{n_{i}})_{i\geq 1}$, such that $(\mathbf{u}%
_{n_{i}})_{i\geq 1}$ converges in $D$ weakly to $\mathbf{u}$ and such that $(%
\mathbf{\Sigma }_{n_{i}})_{i\geq 1}$ converges in $\mathbb{R}^{2}$ strongly
to $\mathbf{\Sigma }$. By Hardy's inequality the sequence $(\mathbf{u}%
_{n_{i}})_{i\geq 1}$ is bounded in $H^{1}(S_{\varepsilon })$ . {We can
therefore extract a subsequence which converges in $L^{2}(S_{\varepsilon })$
strongly to $\mathbf{u}$. Since $\mathbf{u}_{n}=-\mathbf{e}_{1}$, for all $%
n\in \mathbb{N},$ we find that $\mathbf{u}\in {D}_{-\mathbf{e}%
_{1}}^{\varepsilon }$.}

\bigskip

\noindent $(ii)$ Given $\mathbf{w}\in {\mathcal{D}^{\varepsilon }}$ there
exists $n_{\mathbf{w}}>0$, such that $\mathbf{w}\in {D}^{\varepsilon ,n}$
for all $n\geq n_{\mathbf{w}}$. Therefore, we have for $i$ sufficiently large%
\begin{equation}
\int_{\Omega _{+}}\nabla \mathbf{u}_{n_{i}}:\nabla \mathbf{w}~d\mathbf{x}%
+\int_{\Omega _{+}}\left[ (\mathbf{u}_{n_{i}}+\mathbf{e}_{1})\cdot \nabla 
\mathbf{u}_{n_{i}}\right] \cdot \mathbf{w}~d\mathbf{x}=-\mathbf{\Sigma }%
_{n_{i}}\cdot \Gamma (\mathbf{w})~.  \label{jw}
\end{equation}%
Since $H^{1}(\Delta _{n_{\mathbf w}})$ is compactly imbedded in $L^{4}(\Delta _{n_{\mathbf w}}),$ we
find using Lemma \ref{lem_trilinear}, that (\ref{jw}) remains valid in the
limit. This shows that 
\begin{equation}
\int_{\Omega _{+}}\nabla \mathbf{u}:\nabla \mathbf{w}~d\mathbf{x}%
+\int_{\Omega _{+}}\left[ (\mathbf{u}+\mathbf{e}_{1})\cdot \nabla \mathbf{u}%
\right] \cdot \mathbf{w}~d\mathbf{x}=-\mathbf{\Sigma }\cdot \Gamma (\mathbf{w%
})~.  \label{c1}
\end{equation}%
In the weak limit we have moreover that 
\begin{equation}
\int_{\Omega _{+}}|\nabla \mathbf{u}|^{2}~d\mathbf{x}\leq
\liminf_{i\rightarrow \infty }\int_{\Omega _{+}}|\nabla \mathbf{u}%
_{n_{i}}|^{2}~d\mathbf{x}=\liminf\limits_{i\rightarrow \infty }\mathbf{%
\Sigma }_{n_{i}}\cdot \mathbf{e}_{1}=\mathbf{\Sigma }\cdot \mathbf{e}_{1}~.
\label{c2}
\end{equation}%
Combining (\ref{c1}) and (\ref{c2}) we conclude that there exists $\mathbf{%
\Sigma }\in \mathbb{R}^{2}$ such that, for all $\mathbf{w}\in { {\mathcal{D}}^{\varepsilon}}$,%
\begin{equation}
\int_{\Omega _{+}}\nabla \mathbf{u}:\nabla \mathbf{w}~d\mathbf{x}%
+\int_{\Omega _{+}}\left[ (\mathbf{u}+\mathbf{e}_{1})\cdot \nabla \mathbf{u}%
\right] \cdot \mathbf{w}~d\mathbf{x}=-\mathbf{\Sigma }\cdot \Gamma (\mathbf{w%
})~,
\end{equation}%
and that 
\begin{equation}
\int_{\Omega _{+}}|\nabla \mathbf{u}|^{2}~d\mathbf{x}\leq {\mathbf{\Sigma }}%
\cdot \mathbf{e}_{1}~.
\end{equation}%
This completes the proof of {\bf Theorem \ref{thm_existencews}}.

\subsubsection{\label{Lemma8}Proofs of Lemma \protect\ref{lemexistence} and
Lemma \protect\ref{lemuniformestimate}}

In these proofs $n\in \mathbb N$ is fixed.
Since $D^{\varepsilon ,n}$ and $D_{0}^{\varepsilon ,n}$ are closed subspaces
of ${D}^{\varepsilon }$, they are Hilbert spaces with respect to the scalar
product (\ref{sp}). { The space  $D_{-\mathbf{e}_{1}}^{\varepsilon ,n}$ is
not empty.  Indeed, let $\chi $ be a smooth cut-off function that is equal to one
outside the disk ${B(}2h/3)$ and equal to zero inside the disk ${B(}h/3)$.
As $\overline{S_{\varepsilon }}\subset {B(}h/3)$ the function $\mathbf{\tilde{U}}_{-\mathbf{e}_{1}}(x,y)=-\nabla ^{\bot
}((1-\chi )~y)$ satisfies  $\mathbf{\tilde{U}}_{-\mathbf{e}
_{1}} \in D_{-\mathbf{e}_{1}}^{\varepsilon ,n}.$ We note that $D_{-\mathbf{e}_{1}}^{\varepsilon ,n}$ is an affine
subspace of $D^{\varepsilon ,n}$ with direction   $D^{\varepsilon,n}_{0}.$
For technical reason (see \eqref{fe}), we also introduce $\mathbf{U}_{-\mathbf{e}_{1}}$
the unique minimizer of the $D$-norm amongst the velocity-fields in $D^{\varepsilon,n}_{-\mathbf{e}_1}.$
This  velocity field satisfies :
\begin{enumerate}
\item $D^{\varepsilon,n}_{-\mathbf{e}_1} = \mathbf{U}_{-\mathbf{e}_{1}} + D^{\varepsilon,n}_{0}$
\item $((\mathbf{U}_{-\mathbf{e}_{1}},\mathbf{w})) = 0$ for all velocity fields $\mathbf{w} \in D^{\varepsilon,n}_{0}.$
\end{enumerate}
}
We now reformulate the existence of an approximate weak solution on $\Delta
_{n}$ as a fixed point problem for a functional equation. First, we note
that {\bf Lemma \ref{lem_trilinear}} implies that for all $\mathbf{u}\in D_{-%
\mathbf{e}_{1}}^{\varepsilon ,n}$, the map%
\begin{equation*}
\begin{array}{rcl}
D_{0}^{\varepsilon ,n} & \rightarrow & \mathbb{R}~ \\ 
\mathbf{w} & \mathbf{\mapsto } & \displaystyle{\int_{\Omega _{+}}}\left[ (%
\mathbf{u}+\mathbf{e}_{1})\cdot \nabla \mathbf{u}\right] \cdot \mathbf{w}~d%
\mathbf{x}~,%
\end{array}%
\end{equation*}%
is a continuous linear form. By the Riesz-Fr\'{e}chet theorem we can
therefore define a continuous map $b_{n}^{\ast }$ from $D_{-\mathbf{e}%
_{1}}^{\varepsilon ,n}$ to $D_{0}^{\varepsilon ,n}$ by the formula 
\begin{equation}
((b_{n}^{\ast }(\mathbf{u}),\mathbf{w}))=\int_{\Omega _{+}}\left[ (\mathbf{u}%
+\mathbf{e}_{1})\cdot \nabla \mathbf{u}\right] \cdot \mathbf{w}~d\mathbf{x}%
~,\quad \forall \,\mathbf{w}\in {D}_{0}^{\varepsilon ,n}~.  \label{bns}
\end{equation}%
With these definitions we find, on the one hand, that
$\mathbf{u}$ is  an approximate weak solution on $\Delta _{n}$ if and only
if $\mathbf{u}=\mathbf{U}_{-\mathbf{e}_{1}}+\mathbf{v}$, with {$\mathbf{v}$
a solution of the functional equation 
\begin{equation}
\mathbf{v}=b_{n}^{\ast }(\mathbf{U}_{-\mathbf{e}_{1}}+\mathbf{v})~,\qquad 
\mathbf{v}\in D_{0}^{\varepsilon ,n}~.  \label{fe}
\end{equation}%
On the other hand, \eqref{eq_trilinear} together with \eqref{eq_trilinearsym}
imply that $b_{n}^{\ast }$ is continuous on $D_{-\mathbf{e}%
_{1}}^{\varepsilon ,n}$ equipped with the $L^{4}(\Delta _{n})$-norm. Using
that $H_{0}^{1}(\Delta _{n})$ is compactly imbedded in $L^{4}(\Delta
_{n})$ yields that $b_{n}^{\ast }$ is completely continuous, \emph{i.e.}%
, for any given bounded sequence $(\mathbf{v}_{i})_{i\geq 1}$ in ${D}%
_{0}^{\varepsilon ,n}$, there exists a subsequence $(\mathbf{v}%
_{i_{j}})_{j\geq 1}$ such that the sequence $(b_{n}^{\ast }(\mathbf{U}_{-%
\mathbf{e}_{1}}+\mathbf{v}_{i_{j}}))_{j\geq 1}$ converges strongly in ${D}%
_{0}^{\varepsilon ,n}$. Hence, the Leray-Schauder fixed point theorem (see 
\cite{Kavian93} or \cite[Theorem 11.6, p. 286]{Gilbarg&Trudinger01} for more
details) 
guarantees the existence of a solution of \eqref{fe} by proving a suitable estimate on \emph{a priori} solutions to an auxiliary
problem. This estimate is the content of the following proposition.}

\begin{proposition}
\label{both}There exist constants $\varepsilon _{1}>0$ and $C<\infty $, such
that for all $\varepsilon <\varepsilon _{1}$, $\lambda \in \lbrack 0,1]$ and
all $(\mathbf{u},\mathbf{\Sigma })\in {D_{-\mathbf{e}_{1}}^{\varepsilon ,n}}%
\times \mathbb{R}^{2}$ which satisfy 
\begin{equation}
\int_{\Omega _{+}}\nabla \mathbf{u}:\nabla \mathbf{w}+\lambda \int_{\Omega
_{+}}\left[ (\mathbf{u}+\mathbf{e}_{1})\cdot \nabla \mathbf{u}\right] \cdot 
\mathbf{w}=-\mathbf{\Sigma }\cdot \Gamma (\mathbf{w})~,\quad \forall \ \mathbf{%
w}\in {D}^{\varepsilon ,n}~,  \label{eq_nesttest}
\end{equation}%
we have the bound $\left\Vert \mathbf{u};D\right\Vert +\left\vert \mathbf{%
\Sigma }\right\vert \leq C$.
\end{proposition}

{Because of the Leray-Schauder theory, this {lemma}} implies {\bf Lemma \ref%
{lemexistence}}. Then, {\bf Lemma \ref{lemuniformestimate}} is proved assuming 
$\varepsilon <\varepsilon_{1} $ and applying this proposition to the
constructed approximate solution (in this case $\lambda=1$).%

\bigskip

\begin{proof}[Proof of Proposition \protect\ref{both}]
First we note that given $(\mathbf{u},\mathbf{\Sigma },n,\lambda )$ as in
{\bf Proposition \ref{both}} we can set $\mathbf{w}=\mathbf{u}$ in (\ref%
{eq_nesttest}), and we obtain (\ref{eq_NRJtrunc}). Hence, it suffices to
find a bound on $\mathbf{\Sigma }$. For this purpose, we introduce an
additional family of cut-off functions $\chi _{\delta }$. This family
truncates in balls around the point $(0,1+h)$. Namely, let $\zeta \colon 
\mathbb{R}\rightarrow \mathbb{R}$ be a smooth function such that 
\begin{equation}
{\zeta }(s)=1~,\quad \forall \ s<0~,\quad {\zeta }(s)=0~,\quad \forall \
s>1~.  \label{zeta}
\end{equation}%
Then, given $0<\delta <h/3$, we set for $(x,y)\in \Omega _{+}$,%
\begin{equation}
\chi _{\delta }(x,y)={\zeta }\left( \dfrac{|(x,y-1-h)|}{\delta }-1\right) ~.
\label{eq_nruncfunction}
\end{equation}%
With this definition, we have $\chi _{\delta }=1$ in ${B(}\delta )$ while $%
\chi _{\delta }=0$ in the exterior of ${B(}2\delta )$. Now, given $(\mathbf{u%
},\mathbf{\Sigma },n,\lambda )$ and an obstacle $S_{\varepsilon }$, we set {$%
\delta (\varepsilon )=\lambda _{0}\ {\varepsilon },$ with $\lambda _{0}=\sup
\{|(x,y)|,~(x,y)\in S\}$}, and define, for arbitrary $\mathbf{W}\in \mathbb{R%
}^{2}$, the test-function: 
\begin{equation}
\mathbf{w}_{\varepsilon }=-\mathbf{\nabla }^{\bot }\left( \chi _{\delta
(\varepsilon )}(x,y)~\left[ \mathbf{W}^{\bot }\cdot ((x,y)-(0,1+h))\right]
\right) ~.  \label{cf}
\end{equation}%
{Since $S_{\varepsilon }$ tends homothetically to a point when $\varepsilon
\rightarrow 0,$ we can choose $\varepsilon _{0}$ (say $\varepsilon
_{0}=h/(3\lambda _{0})$ for instance) such that $\mathbf{w}_{\varepsilon }$
is equal to $\mathbf{W}$ on $S_{\varepsilon }$ and equal to zero outside 
$B(2h/3)$ for $\varepsilon <\varepsilon _{0}.$} Thus, we can use $\mathbf{w}%
_{\varepsilon }$ as a test-function in (\ref{eq_nesttest}). By construction
of $\mathbf{w}_{\varepsilon }$, there exists a { universal} constant $C_{1}$ such that 
\begin{equation}
\Vert \mathbf{w}_{\varepsilon };D\Vert +\Vert \mathbf{w}_{\varepsilon
};L^{\infty }(\mathbb{R}^{2})\Vert \leq C_{1}|\mathbf{W}|~,  \label{web}
\end{equation}%
and we get from (\ref{eq_nesttest}) the inequality%
\begin{eqnarray*}
|\mathbf{\Sigma }\cdot \mathbf{W}| &\leq &\Vert \mathbf{u};D\Vert ~\Vert 
\mathbf{w}_{\varepsilon };D\Vert +\lambda \Vert \mathbf{u}+\mathbf{e}%
_{1};L^{2}({B(}2\delta (\varepsilon )))\Vert ~\Vert \mathbf{u};D\Vert ~\Vert 
\mathbf{w}_{\varepsilon };L^{\infty }(\Omega _{+})\Vert  \\
&\leq &C_{1}|\mathbf{W}|~\left( \Vert \mathbf{u};D\Vert +\Vert \mathbf{u}+%
\mathbf{e}_{1};L^{2}({B(}2\delta (\varepsilon )))\Vert ~\Vert \mathbf{u}%
;D\Vert \right) ~.
\end{eqnarray*}%
Since $\mathbf{u}_{|_{S_{\varepsilon }}}=-\mathbf{e}_{1}$, Poincar\'{e}'s
inequality implies that there exists a constant $\widetilde{C_{2}}~$ such
that 
\begin{equation*}
\Vert \mathbf{u}+\mathbf{e}_{1};L^{2}({B(}{ 2\delta(\varepsilon)} ))\Vert \leq 
\widetilde{C_{2}}~\Vert \mathbf{u};D\Vert ~.
\end{equation*}%
A scaling argument shows $\widetilde{C_{2}}=\varepsilon C_{2}$ with a
constant $C_{2}$ independent of $\varepsilon $ and $\mathbf{u}$ {(see \cite[Exercise 4.10]{Galdi94} for a construction
of $C_2$)}. Therefore, 
\begin{equation}
|\mathbf{\Sigma }|\leq C_{1}\left( \Vert \mathbf{u};D\Vert +C_{2}\varepsilon
\Vert \mathbf{u};D\Vert ^{2}\right) ~.  \label{eq_Si}
\end{equation}%
From (\ref{eq_Si}) and (\ref{eq_NRJtrunc}) we find that {if $\varepsilon $
satisfies moreover $\varepsilon <1/(2C_{1}C_{2})$}, we have a bound on $|%
\mathbf{\Sigma }|$ and $\Vert \mathbf{u};D\Vert $ which is independent of $n$%
, $\lambda $, and $\varepsilon $, namely 
\begin{equation}
\Vert \mathbf{u};D\Vert \leq |\mathbf{\Sigma }|^{1/2}\leq 2C_{1}~.
\label{un1}
\end{equation}
\end{proof}

\subsection{Weak solutions for obstacles of vanishing size}

From now on, we choose once and for all $(\mathbf{u}_{\varepsilon
})_{\varepsilon <\varepsilon _{1}}$ a bounded family of $D$ such that $%
\mathbf{u}_{\varepsilon }$ is a weak solution of Problem 1 for $%
S_{\varepsilon }$ for all $\varepsilon <\varepsilon _{1}$. Such a sequence
exists according to {\bf Theorem \ref{thm_existencews}}. In this section we 
complete the proof of {\bf Theorem \ref{thm_cvg0main}} by 
showing that the sequence $(\mathbf{u}_{\varepsilon })_{\varepsilon <\varepsilon
_{1}}$ converges to zero when the size of the obstacle tends to zero. As a
by-product, we also obtain the pressure ${p}$ associated with a weak solution.

The convergence is proved in the family of spaces $(\mathcal{C}^{m}(\mathcal{%
A}_n))_{(m,n) \in \mathbb{N}^2}$. We refer the reader to {Section \ref%
{sec_functionspaces} for the definition of sets $({\mathcal{A}}_n)_{n\in 
\mathbb{N}}$.} We recall here that they satisfy the following fundamental
properties

\begin{itemize}
\item for all $n \in \mathbb{N},$ $\mathcal{A}_n \subset \overline{\mathcal{A%
}_n} \subset \mathcal{A}_{n+1},$

\item for all $n \in \mathbb{N},$ $\mathcal{A}_n$ has a smooth boundary,

\item $\bigcup_{n\in \mathbb{N}} \mathcal{A}_n = \Omega_+ \setminus (0,1+h).$
\end{itemize}

\textbf{Theorem \ref{thm_cvg0main}} is a straightforward consequence of the
following two lemmas which we prove in the following subsections.

\begin{lemma}
\label{lem_cvg0aux} Given $\eta >0$ there exists {$0<\varepsilon_\eta $}
such that $\Vert \mathbf{u}_{\varepsilon };D\Vert \leq \eta $ for all $%
\varepsilon <{ \varepsilon_\eta }$.
\end{lemma}

{
\begin{lemma}
\label{lem_bootstrap}Let $(n,m)\in \mathbb{N}^2$ and $\varepsilon < \varepsilon_1 $ such that ${ S_{\varepsilon }\subset \subset [ \ \mathbb  R^2_+ \setminus \overline{\mathcal{A}_{n+m+1}}} \ ]$. 
Then, there exists a
constant $C_{m,n}$, depending only on $m$ and $n,$ for which any weak
solution $\mathbf{u}$ of Problem 1 for $S_{\varepsilon }$ such that $\Vert 
\mathbf{u};D\Vert \leq 1$ satisfies 
\begin{enumerate}
\item[(i)] there exists a pressure $p$ such that $(\mathbf{u},p)$ is solution to \eqref{Problem2} 
\item[(ii)] the following estimate holds true 
\begin{equation}
\Vert \mathbf{u};H{^{m+1}(}\mathcal{A}_{n}{)}\Vert +\Vert p;H{^{m}(\mathcal{A%
}_{n})/\mathbb{R}}\Vert \leq C_{m,n}~\Vert \mathbf{u};D\Vert ~.
\label{eq_estCm1}
\end{equation}
\end{enumerate}
\end{lemma}
}

\bigskip

\noindent \textit{Remark:}

\noindent One might be tempted to assume that {the smallness estimate} of
{\bf Theorem \ref{thm_cvg0main}} is straightforward, since the fluid is moving
only due to the no-slip boundary condition on $\partial S_{\varepsilon }$.
The smaller the body, the smaller should be the fluid flow which is induced
by this boundary condition so that the flow should be zero in the limit
of a body of vanishing size. The following scaling argument shows that,
because of the Stokes paradox, things are not quite as simple. Let $(\mathbf{%
u}_{\varepsilon })_{\varepsilon >0}$ a family of weak solutions for $%
S_{\varepsilon }$ and $\Omega _{+}^{\varepsilon }=\{(x,y)\in \mathbb{R}%
^{2}\left\vert {}\right. (\varepsilon x,\varepsilon y-(h+1))\in \Omega
_{+}\} $, and let $\mathbf{v}_{\varepsilon }(x,y)=\mathbf{u}_{\varepsilon
}(\varepsilon x,\varepsilon y-(h+1))$ and ${q_{\varepsilon }}%
(x,y)=p_{\varepsilon }(\varepsilon x,\varepsilon y-(h+1))$. This scaling
does not affect the $D-$norm, so that $\Vert \nabla \mathbf{v}_{\varepsilon
};L^{2}(\Omega _{+}^{\varepsilon })\Vert =\Vert \nabla \mathbf{u}%
_{\varepsilon };L^{2}({\Omega }_{+})\Vert $. Therefore, {if} the { family} $(%
\mathbf{u}_{\varepsilon })_{\varepsilon >0}$ is bounded in $D$, the { family} 
$(\Vert \nabla \mathbf{v}_{\varepsilon };L^{2}(\Omega _{+}^{\varepsilon
})\Vert )_{\varepsilon >0}$ is also bounded, and the functions $\mathbf{v}%
_{\varepsilon }$ satisfy in $\Omega _{+}^{\varepsilon }$ the equation%
\begin{equation}
\begin{array}{r}
-\varepsilon \left( \mathbf{v}_{\varepsilon }\cdot \nabla \right) \mathbf{v}%
_{\varepsilon }-\varepsilon \partial _{x}\mathbf{v}_{\varepsilon }+\Delta 
\mathbf{v}_{\varepsilon }-{\nabla q_{\varepsilon }}=0~, \\ 
\nabla \cdot \mathbf{v}_{\varepsilon }=0~,%
\end{array}%
\end{equation}%
with$\,$the boundary condition $\mathbf{v}_{\varepsilon }$ $=-\mathbf{e}_{1}$
on $\partial S_{1}$. Using the same line of arguments as in the previous
section we can therefore extract a subsequence converging in the topology
induced by the $D-$norm to some function $\mathbf{v}$ for which $\Vert
\nabla \mathbf{v};L^{2}(\mathbb{R}^{2})\Vert $\textbf{\ }is finite and which
solves the Stokes equations in $\mathbb{R}^{2}\setminus \overline{S_{1}}$
with the boundary condition $\mathbf{v}$ $=-\mathbf{e}_{1}$ on $\partial
S_{1}$. By the Stokes paradox this implies that $\mathbf{v=}-\mathbf{e}_{1}$
(see \cite[Theorem 2.2 p. 253]{Galdi98}), and therefore the sequence $%
\mathbf{v}_{\varepsilon }$ does not converge to zero with $\varepsilon $.
Note that this remark does not contradict {\bf Theorem \ref{thm_cvg0main}}. It
only means that, when $\varepsilon $ goes to zero, $\mathbf{v}_{\varepsilon
} $ does take values close to $-\mathbf{e}_{1}$ on a part of the domain that
increases in size and covers eventually all of $\mathbb{R}^{2}$. The
diameter of this region remains however small compared with $1/\varepsilon $%
, and its size therefore converges to zero in the un-scaled variables.

\subsubsection{Proof of Lemma \protect\ref{lem_cvg0aux}}

We prove {\bf Lemma \ref{lem_cvg0aux}} by contradiction. First, we assume that there exists $%
\eta _{0}>0$, sequences $(\varepsilon _{i})_{_{i\in \mathbb{N}}}\in
(0,\varepsilon _{1})^{\mathbb{N}}$ and $(\mathbf{u}_{i},\mathbf{\Sigma }%
_{i})_{i\in \mathbb{N}}\in (D\times \mathbb{R}^{2})^{\mathbb{N}}$ such that $%
\lim\limits \varepsilon _{i}=0$, that $\mathbf{u}_{i}=%
\mathbf{u}_{\varepsilon _{i}}$ has 
associated force $\mathbf{\Sigma }_{i}$ and is such that $\Vert \mathbf{u}_{i};D\Vert
\geq \eta _{0}$ for all $i\in \mathbb{N}$. Then, the sequence $(\mathbf{u}%
_{i},\mathbf{\Sigma }_{i})_{i\in \mathbb{N}}$ is bounded. This implies the
existence of a pair $(\mathbf{u},\mathbf{\Sigma })\in D\times \mathbb{R}^{2}$
and of a subsequence $(\mathbf{u}_{i_{j}},\mathbf{\Sigma }_{i_{j}})_{j\in 
\mathbb{N}}$ such that $\mathbf{u}_{i_{j}}\rightharpoonup _{j\rightarrow \infty }%
\mathbf{u}$ weakly in $D$ and such that $\mathbf{\Sigma }_{i_{j}}\rightarrow
_{j\rightarrow \infty }\mathbf{\Sigma }$ strongly in $\mathbb{R}^{2}$.

\medskip

We now proceed as in the proof of {\bf Proposition \ref{both}}. Let $\chi _{\delta
}$ be the cut-off function defined in (\ref{eq_nruncfunction}) and define,
as in (\ref{cf}) for $0<\delta <h/3$ and arbitrary $\mathbf{W}\in \mathbb{R}%
^{2}$ the test-function $\mathbf{w}_{\delta }$ 
\begin{equation*}
\mathbf{w}_{\delta }(x,y)=-\nabla ^{\bot }\left( \chi _{\delta }(x,y)~\left[ 
\mathbf{W}^{\bot }\cdot ((x,y)-(0,1+h))\right] \right) ~,
\end{equation*}%
which has support in ${B(}{ 2h/3})$. As in (\ref{web}) there exists a { universal} constant $%
C_{1}<\infty $ such that 
\begin{equation}
\Vert \mathbf{w}_{\delta };D\Vert +\Vert \mathbf{w}_{\delta };L^{\infty }({%
\Omega }_{+})\Vert \leq C_{1}|\mathbf{W}|~.
\end{equation}%
Since $\lim \varepsilon _{i}=0$, there exists $i_{\delta }$ such that for $%
i\geq i_{\delta }$ the function $\mathbf{w}_{\delta }$ is an admissible
test-function, and we have : 
\begin{equation*}
\int_{\Omega _{+}}\nabla \mathbf{u}_{i}:\nabla \mathbf{w}_{\delta }~d\mathbf{%
x}+\int_{\Omega _{+}}[(\mathbf{u}_{i}+\mathbf{e}_{1})\cdot \nabla \mathbf{u}%
_{i}]\cdot \mathbf{w}_{\delta }~d\mathbf{x}=-\mathbf{\Sigma }_{i}\cdot 
\mathbf{W}~.
\end{equation*}%
In the limit as $i$ goes to infinity we therefore get 
\begin{equation*}
\int_{\Omega _{+}}\nabla \mathbf{u}:\nabla \mathbf{w}_{\delta }~d\mathbf{x}%
+\int_{\Omega _{+}}[(\mathbf{u}+\mathbf{e}_{1})\cdot \nabla \mathbf{u}]\cdot 
\mathbf{w}_{\delta }~d\mathbf{x}=-\mathbf{\Sigma }\cdot \mathbf{W}~,
\end{equation*}%
so that,%
\begin{equation*}
|\mathbf{\Sigma }|\leq C_{1}\Vert \nabla \mathbf{u};L^{2}({B(}{ 2\delta/3}))\Vert %
\left[ 1+\Vert \mathbf{u}+\mathbf{e}_{1};L^{2}({B(}{ 2\delta/3}))\Vert \right] ~.
\end{equation*}%
Letting $\delta $ go to $0$, yields $\mathbf{\Sigma }=0$,\textit{\ i.e.}, $%
\mathbf{\Sigma }_{i}\rightarrow _{i\rightarrow \infty }0$ which by the
energy estimate (\ref{eq_NRJ}) implies that $\lim \Vert \mathbf{u}%
_{i};D\Vert =0$, in contradiction with our assumption.

\subsubsection{\label{sec_WS3}Proof of Lemma \protect\ref{lem_bootstrap}}
Let $(n,m)\in \mathbb{N}^{2}$ and $\varepsilon ,\mathbf{u}$ be given as in %
{\bf Lemma \ref{lem_bootstrap}}.{\ At first, we recall how to construct the
pressure associated with $\mathbf u$.} We test (\ref{eq_wf}) with smooth divergence free vector-fields having
compact support in $\Omega _{+}\setminus \overline{S_{\varepsilon }}$. This
shows that $\mathbf{u}$ is a generalized solution in the sense of {\cite[Definition IV.1.1, p. 185]{Galdi94}} of the Stokes equation in $\Omega
_{+}\setminus \overline{S_{\varepsilon }}$ with source term%
\begin{equation*}
\mathbf{f}=(\mathbf{u}+\mathbf{e}_{1})\cdot \nabla \mathbf{u}~.
\end{equation*}%
Since, for all $\Omega ^{\prime }\subset \subset (\Omega _{+}\setminus 
\overline{S_{\varepsilon }})$, we have $\mathbf{f}\in H^{-1}(\Omega ^{\prime
})$ with the bound 
\begin{equation}
\Vert \mathbf{f};H^{-1}(\Omega ^{\prime })\Vert \leq C(\Omega ^{\prime })%
\left[ \Vert \mathbf{u};H^{1}(\Omega ^{\prime })\Vert^2 +\Vert \mathbf{u}%
;D\Vert \right]   \label{eq_controlf}
\end{equation}%
we can apply { \cite[lemma IV.1.1,p. 186]{Galdi94}} to construct a function $%
p\in L_{loc}^{2}(\Omega _{+}\setminus \overline{S_{\varepsilon }})$ such
that, in the sense of distributions, 
\begin{equation}
\left\{ 
\begin{array}{rcl}
\Delta \mathbf{u}-\nabla p & = & \mathbf{f}~, \\ 
\nabla \cdot \mathbf{u} & = & 0~,%
\end{array}%
\right. 
\end{equation}%
in $\Omega _{+}\setminus \overline{S_{\varepsilon }}$.  Classically, this pressure $p$ is unique up to
a finite number of  constants (equal to the number of connected components of $\Omega_+ \setminus \overline{S_{\varepsilon}}$). We call $p$ the \textit{pressure associated with }$\mathbf{u}$ and we indeed have that $(\mathbf{u},p)$ satisfies \eqref{Problem2}.

\bigskip

The remainder of {\bf Lemma \ref{lem_bootstrap}} is {obtained via an induction
argument (with respect to $m \in \mathbb N$). Namely, we prove that, for all $k\leq m$ the following statement
holds true:\\[4pt]
There exist  constants $C_{m,k}$ depending only on $m$ and $k$ such that : 
\begin{equation}
\Vert \mathbf{u};H{^{k+1}(}\mathcal{A}_{n+m-k}{)}\Vert +\Vert p;H{^{k}(%
\mathcal{A}_{n+m-k})/\mathbb{R}}\Vert \leq C_{m,k}~\Vert \mathbf{u};D\Vert ~.
\tag{${\cal P}_{k}$}  \label{eq_estCm}
\end{equation}%
\bigskip 

\noindent \textit{Proof, initialization:} {The restriction of $(\mathbf{u},p)
$ to $\mathcal{A}_{n+m}$} is a solution of the Stokes equations with
source term $\mathbf{f=}(\mathbf{u}+\mathbf{e}_{1})\cdot \nabla \mathbf{u}$
and boundary data {$\mathbf{u}=\mathbf{u}_{|_{\partial \mathcal{A}_{n+m}}}$%
}. Hence, combining { \cite[theorem 1.1 p. 188]{Galdi94}} and (\ref{eq_controlf}%
), and using that $\Vert \mathbf{u};D\Vert \leq 1$, we find that {%
\begin{equation}
\Vert \mathbf{u};H^{1}(\mathcal{A}_{n+m})\Vert +\Vert p;L^{2}(\mathcal{A}%
_{n+m})/\mathbb{R}\Vert \leq \widetilde{C}_{0}\left[ \Vert \mathbf{u}%
;D\Vert +\Vert \mathbf{u};D\Vert ^{2}\right] \leq C_{m,0}\Vert \mathbf{u}%
;D\Vert ~.  \label{eq_induction0}
\end{equation}%
Our statement holds true for $k=0$.} \bigskip 

{Before the inductive step of the proof, we need to compute an $L^{\infty }$
estimate on $\mathbf{u}$ inside $\mathcal{A}_{n+m}$. To this end,} we recall
that we have by construction that $\mathcal{A}_{n+m}\subset ( \: \overline{%
\mathcal{A}_{n+m}} \: \cap \: \Omega_+) \subset \mathcal{A}_{n+m+1}$. Hence there exists a smooth
truncation function $\chi \in \mathcal{C}^{\infty}(\overline{\Omega}_+)$, such that $\chi=1$ on $\mathcal{A}_{n+m}$ and $\chi=0$ outside $%
\mathcal{A}_{n+m+1}$.  
{ We make the dependance of $\chi$ upon $n$ implicit for legibility.}
Let $\tilde{\mathbf{u}}=\chi\mathbf{u}$ and $%
\tilde{p}=\chi p$.
Then $(\tilde{\mathbf{u}},\tilde{p})$ is a solution
of the Stokes system 
\begin{equation}
\left\{ 
\begin{array}{rl}
\Delta \tilde{\mathbf{u}}-\nabla \tilde{p}=\widetilde{\mathbf{f}} & \text{ on }%
\mathcal{A}_{n+m+1}~\text{$,$} \\ 
\nabla \cdot \tilde{\mathbf{u}}=\tilde{g} & \text{ on }\mathcal{A}_{n+m+1}~%
\text{$,$}%
\end{array}%
\quad \right. 
\begin{array}{c}
\text{with }\tilde{\mathbf{u}}=0\text{ on $\partial \mathcal{A}_{n+m+1}~,$}%
\end{array}
\label{eq_stokestilde}
\end{equation}%
where%
\begin{equation}
\left\{
\begin{array}{rcl}
\widetilde{\mathbf{f}} & = & \chi (\mathbf{u}+\mathbf{e}_{1})\cdot \nabla 
\mathbf{u}+2\nabla \chi \cdot \nabla \mathbf{u}+\left( \Delta \chi \right) 
\mathbf{u}-p\nabla \chi ~, \\ 
\tilde{g} & = & \mathbf{u}\cdot \nabla \chi ~.%
\end{array}
\right.
\label{fg}
\end{equation}%
Using the bound (\ref{eq_induction0}) on $(\mathbf{u},p)$, we find that, { for a given $q<2$ (say $q=3/2$), there holds} $%
\widetilde{\mathbf{f}}\in L^{q}(\mathcal{A}_{n+m+1})$ and $\tilde{g}\in
W^{1,q}(\mathcal{A}_{n+m+1})$. Furthermore we have, 
\begin{multline*}
\Vert \widetilde{\mathbf{f}};L^{q}(\mathcal{A}_{n+m+1})\Vert +\Vert \tilde{g}%
;W^{1,q}(\mathcal{A}_{n+m+1})\Vert  \\[6pt]
\leq C_{m,q}\left[ \Vert \mathbf{u};H^{1}(\mathcal{A}_{n+m+1})\Vert +\Vert 
\mathbf{u};H^{1}(\mathcal{A}_{n+m+1})\Vert ^{2}+\Vert p;{L^{2}(\mathcal{A}%
_{n+m+1})}\Vert \right] ~.
\end{multline*}%
Applying { \cite[Exercise IV.6.2, p. 232]{Galdi94}} {we get} that $\tilde{%
\mathbf{u}}\in W^{2,q}(\mathcal{A}_{n+m+1})$ and $\tilde{p}\in W^{1,q}(%
\mathcal{A}_{n+m+1})$, and that
\begin{multline}
\Vert \tilde{\mathbf{u}};W^{2,q}(\mathcal{A}_{n+m+1})\Vert +\Vert \tilde{p}%
;W^{1,q}(\mathcal{A}_{n+m+1})/\mathbb{R}\Vert  \\[6pt]
\leq C_{m,q}\left[ \Vert \mathbf{u};H^{1}(\mathcal{A}_{n+m+1})\Vert +\Vert 
\mathbf{u};H^{1}(\mathcal{A}_{n+m+1})\Vert ^{2}+\Vert p;{L^{2}(\mathcal{A}%
_{n+m+1})}\Vert \right] ~.  \label{but}
\end{multline}%
Note that we can always replace $p$ by $p+c$ before truncation, so that we
can replace $\Vert p;L^{2}(\mathcal{A}_{n+m+1})\Vert $ by $\Vert p;L^{2}(%
\mathcal{A}_{n+m+1})/\mathbb{R}\Vert $ in the right-hand side of the last
inequality, as well as in the estimates that follow. Combining (\ref{but})
with (\ref{eq_induction0}), we get 
\begin{equation}
\Vert \tilde{\mathbf{u}};W^{2,q}(\mathcal{A}_{n+m+1})\Vert +\Vert \tilde{p}%
;W^{1,q}(\mathcal{A}_{n+m+1})/\mathbb{R}\Vert \leq C_{m,q}\Vert \mathbf{u}%
;D\Vert ~.
\end{equation}%
Therefore, we have in particular that ${\mathbf{u}}\in W^{2,q}(\mathcal{A}%
_{n+m})\subset L^{\infty }(\mathcal{A}%
_{n+m})$ with 
$$\Vert \mathbf{u};L^{\infty }(%
\mathcal{A}_{n+m}))\Vert \leq K_{m,n}\Vert \mathbf{u};D\Vert~. $$

\bigskip

{\noindent \textit{Proof, inductive step:}\newline
Assuming that for $k\leq m,$ there exist 
constants $C_{m,k}$ depending only on $m$ and $k$ such that : 
\begin{equation*}
\Vert \mathbf{u};H{^{k+1}(}\mathcal{A}_{n+m-k}{)}\Vert +\Vert p;H{^{k}(%
\mathcal{A}_{n+m-k})/\mathbb{R}}\Vert \leq C_{m,k}~\Vert \mathbf{u};D\Vert ~,
\end{equation*}%
we apply again the same truncation technique as described above. Namely, we introduce $\chi \in \mathcal{C}^{\infty}(\overline{\Omega}_+)$ a smooth truncation
function such that $\chi=1$ on $\mathcal{A}_{n+m-k-1}$ and $\chi=0$
outside $\mathcal{A}_{n+m-k}$ and we let $\tilde{\mathbf{u}}=\chi%
\mathbf{u}$ and $\tilde{p}=\chi p$. Then $(\tilde{\mathbf{u}},\tilde{p})$
is a solution of the Stokes system (\ref{eq_stokestilde}), on $\mathcal{A}%
_{n+m-k}$ with homogeneous boundary condition. Hence, we get by the
ellipticity of the Stokes operator that 
\begin{equation}
\Vert \tilde{\mathbf{u}};{H^{k+2}(\mathcal{A}_{n+m-k})}\Vert +\Vert \tilde{%
p};{H^{k+1}(\mathcal{A}_{n+m-k})}/\mathbb{R}\Vert   
\leq \tilde{C}_{m,k}\left[ \Vert \mathbf{\tilde{f}};{H^{k}(\mathcal{A}%
_{n+m-k}))}\Vert +\Vert \tilde{g};{H^{k+1}(\mathcal{A}_{n+m-k})}\Vert \right]
~.
\end{equation}%
We also have%
\begin{multline*}
\Vert \mathbf{\tilde{f}};{H^{k}(\mathcal{A}_{n+m-k}))}\Vert +\Vert \tilde{g};%
{H^{k+1}(\mathcal{A}_{n+m-k})}\Vert  \\
\leq \tilde{C}_{m,k}\Big[\Vert \mathbf{u};L^{\infty }(\mathcal{A}%
_{n+m})\Vert ~\Vert \mathbf{u};{H^{k+1}(\mathcal{A}_{n+m-k})}\Vert +\Vert 
\mathbf{u};{H^{k+1}(\mathcal{A}_{n+m-k})}\Vert ^{2} \\
+\Vert \mathbf{u};{H^{k+1}(\mathcal{A}_{n+m-k})}\Vert +\Vert p;{H^{k}(%
\mathcal{A}_{n+m-k})/\mathbb{R}}\Vert \Big]~,
\end{multline*}%
and therefore there exists, by the induction assumption, a constant ${C}%
_{m,k+1}$, such that%
\begin{equation*}
\Vert {\mathbf{u}};{H^{k+2}(\mathcal{A}_{n+m-k-1})}\Vert +\Vert {p};{H^{k+1}(%
\mathcal{A}_{n+m-k-1})}/\mathbb{R}\Vert \leq C_{m,k+1}~\Vert \mathbf{u};{D}%
\Vert ~.
\end{equation*}%
This completes the inductive step and ends the proof.}

\section{Behavior of weak solutions at large distance from the obstacle.}

\label{sec_alphasolution} {In this section we show that the weak solutions of Problem~1 
constructed above decay at infinity with the expected rate. Namely, we prove: }

\begin{theorem}
\label{mainexistencetheorem} There exists $\varepsilon _{e}>0$, such that,
for all $\varepsilon <\varepsilon _{e},$ the weak solution $\mathbf{u}%
_{\varepsilon }$ satisfies the decay estimate, 
\begin{equation}
|\mathbf{u}_{\varepsilon }(x,y)|\leq \dfrac{C_{\varepsilon }}{y^{\frac{3}{2}}%
}~.\quad \forall \,(x,y)\in \Omega _{+}\setminus \overline{S_{\varepsilon }}%
~.  \label{decayestimate}
\end{equation}%
for some $C_{\varepsilon }<\infty .$
\end{theorem}

This result is proved in three steps by comparing weak solutions with $%
\alpha $-solutions. First, we show how to construct solutions for Problem$~$%
2 by truncating a weak solution for Problem$~$1. We prove in particular
that, when the solid is sufficiently small, 
weak solutions to Problem~1 provided by {\bf Theorem \ref{thm_cvg0main}} yield weak
solutions to Problem~2 with a source term which 
is arbitrary small, so that we are able to construct $\alpha$-solutions. We conclude by proving that 
any weak solution coincides with the $\alpha$-solution when the source-term is sufficiently small.

\subsection{Truncation procedure}

We start this section by describing how to construct a solution for Problem$~
$2 by truncating a weak solution for Problem$~$1. Let 
\begin{equation*}
\begin{array}{rrcl}
\Pi : & D\cap \mathcal{C}^{\infty }(\overline{\Omega _{+}}\setminus {B(}h/4))
& \longrightarrow  & \mathcal{C}^{\infty }(\overline{\Omega _{+}}\setminus {B(}%
h/4)) \\[4pt]
& \mathbf{w} & \longmapsto  & \psi (x,y)=-{ \displaystyle{\int_{1}^{y}}}\mathbf{%
w}(x,z)\cdot \mathbf{e}_{1}~d\text{$z$}~.%
\end{array}%
\end{equation*}%
The divergence-free condition satisfied by $\mathbf{w}$ implies that $\nabla
^{\bot }\Pi \lbrack \mathbf{w}]=\mathbf{w}$, and that 
\begin{equation*}
\Pi \lbrack \mathbf{w}](x,y)=\int_{\gamma }\mathbf{w}^{\bot }\cdot d\gamma ~,
\end{equation*}%
for any path $\gamma $ such that { $\gamma (0)=(0,1)$} and $\gamma (1)=(x,y)$.
Hence, it is sufficient that $\mathbf{w}$ is {smooth} in $\overline{\Omega
_{+}}\setminus {B(}h/3)$ in order for the associated stream-function $\Pi \lbrack \mathbf{w}]$ 
to be smooth in $\overline{\Omega _{+}}\setminus {B(}h/4)$.  {More precisely, for
all $m\in \mathbb{N}$, there exists a constant $C_{m}$, such that, 
\begin{equation}
\Vert \Pi \lbrack \mathbf{w}];{\mathcal{C}^{m}(\overline{{B(}2h/3)}\setminus 
{{B(}h/3)})}\Vert \leq C_{m}\Vert \mathbf{w};{\mathcal{C}^{m-1}(}\overline{{%
{ \mathcal{A}_{2}}}}{)}\Vert ~\quad \forall \,\mathbf{w}\in D\cap \mathcal{C}%
^{\infty }(\overline{\Omega _{+}}\setminus {B(}h/4))~.  \label{estPi}
\end{equation}%
} We introduce a truncation function $\chi \in \mathcal{C}^{\infty }(\mathbb{%
R}^{2})$ which satisfies 
\begin{equation*}
\chi (x,y)=\left\{ 
\begin{array}{ll}
0 & \text{if $|(x,y)-(0,1+h)|<$}h/3 \\ 
\in \lbrack 0,1], & \text{if $|(x,y)-(0,1+h)|\in ({h/3},2h/3)$} \\ 
1 & \text{if $|(x,y)-(0,1+h)|>2h/3$}%
\end{array}%
\right. 
\end{equation*}%
and define truncation operators $\mathbf{T}_{v}$ and ${T}_{\pi }$ for the
velocity and the pressure as follows 
\begin{equation*}
\begin{array}{rrcc}
\mathbf{T}_{v}: & D\cap \mathcal{C}^{\infty }(\overline{\Omega _{+}}%
\setminus {B(}h/4)) & \longrightarrow  & \mathcal{C}^{\infty }(\overline{%
\Omega _{+}}) \\[4pt]
& \mathbf{w} & \longmapsto  & \nabla ^{\bot }\left[ \chi \Pi \lbrack \mathbf{%
w}]\right] 
\end{array}%
\end{equation*}%
and 
\begin{equation*}
\begin{array}{rrcc}
{T}_{\pi }: & \mathcal{C}^{\infty }(\overline{\Omega _{+}}\setminus {B(}h/4))
& \longrightarrow  & \mathcal{C}^{\infty }(\overline{\Omega _{+}}) \\[4pt]
& q & \longmapsto  & \chi q%
\end{array}%
\end{equation*}%
These operators are well-defined, since the truncation function $\chi $
vanishes identically in ${B(}h/4)$. For any $\mathbf{w}\in D\cap \mathcal{C}%
^{\infty }(\overline{\Omega _{+}}\setminus {B(}h/4))$ and $q\in \mathcal{C}%
^{\infty }(\overline{\Omega _{+}}\setminus {B(}h/4))$, we have {by 
a straightforward application of \eqref{estPi}} that
\begin{enumerate}
\item[(T-\emph{i})] $\mathbf{T}_{v}[\mathbf{w}]\in D\cap \mathcal{C}^{\infty
}(\overline{\Omega _{+}})$, and $T_{\pi }[q]\in \mathcal{C}^{\infty }(%
\overline{\Omega _{+}})$,

\item[(T-\emph{ii})] $\mathbf{T}_{v}[\mathbf{w}]=\mathbf{w}$ and $T_{\pi
}[q]=q$ in $\overline{\Omega _{+}}\setminus {B(}2h/3)$,

\item[(T-\emph{iii})] Given $m\in \mathbb{N}$, there exists a constant $C_{m}
$ such that 
\begin{eqnarray*}
\Vert \mathbf{T}_{v}[\mathbf{w}];{\mathcal{C}^{m+1}(\overline{{B(}2h/3)}%
\setminus {{B(}h/3)})}\Vert  &\leq &C_{m}\Vert \mathbf{w};{\mathcal{C}^{m+1}(%
\overline{{ \mathcal{A}_{2}}})}\Vert ~, \\
\Vert T_{\pi }[q];{\mathcal{C}^{m}(\overline{{B(}2h/3)}\setminus {{B(}h/3)})}%
\Vert  &\leq &C_{m}\Vert q;{\mathcal{C}^{m}(\overline{{ \mathcal{A}_{2}}})}%
\Vert ~.
\end{eqnarray*}
\end{enumerate}

{Next for} $(\mathbf{w},q)\in (D \, \cap \, \mathcal{C}^{\infty }(\overline{\Omega
_{+}}\setminus {B(}h/4)))\times \mathcal{C}^{\infty }(\overline{\Omega _{+}}%
\setminus {B(}h/4))$ we define the function $\mathbf{f}\in \mathcal{C}%
^{\infty }(\overline{\Omega _{+}}\setminus {B(}h/4))$ by 
\begin{equation*}
\mathbf{f}=(\mathbf{w}+\mathbf{e}_{1})\cdot \nabla \mathbf{w}-\Delta \mathbf{%
w}+\nabla q~,
\end{equation*}%
and we define $\mathbf{f}=0$ inside ${B(}h/4).$ Finally we define the
function $TNS[\mathbf{w},q]$ on $\Omega _{+}$ by%
\begin{equation*}
TNS[\mathbf{w},q]=-\chi \mathbf{f}+\left[ (\tilde{\mathbf{w}}+\mathbf{e}%
_{1})\cdot \nabla \tilde{\mathbf{w}}-\Delta \tilde{\mathbf{w}}+\nabla \tilde{%
q}\right] ~,
\end{equation*}%
where $(\tilde{\mathbf{w}},\tilde{q})=(\mathbf{T}_{v}[\mathbf{w}],T_{\pi
}[q])$. {Given $(\mathbf{w},q)\in (D\cap \mathcal{C}^{\infty }(\overline{%
\Omega _{+}}\setminus {B(}h/4)))\times \mathcal{C}^{\infty }(\overline{%
\Omega _{+}}\setminus {B(}h/4))$, the above properties of the truncation
operators $\mathbf{T}_{v}$ and $T_{\pi }$ imply 
that the function $TNS[\mathbf{w},q]$ satisfies: }

\begin{enumerate}
\item[(S-\emph{i})] $TNS[\mathbf{w},q]$ {is smooth and} has compact support
in $\overline{{B(}2h/3)}\setminus {B(}h/3)$,

\item[(S-\emph{ii})] The truncated functions $\tilde{\mathbf{w}}=\mathbf{T}%
_{v}[\mathbf{w}]$ and $\tilde{q}=T_{\pi }[q]$ satisfy: 
\begin{equation*}
\left\{ 
\begin{array}{rcll}
(\tilde{\mathbf{w}}+\mathbf{e}_{1})\cdot \nabla \tilde{\mathbf{w}}-\Delta 
\tilde{\mathbf{w}}+\nabla \tilde{q} & = & \chi \mathbf{f} + TNS[\mathbf{w},q]~,
& \text{ in $\Omega _{+}$}~, \\ 
\nabla \cdot \tilde{\mathbf{w}} & = & 0~, & \text{ in $\Omega _{+}$}~.%
\end{array}%
\right. 
\end{equation*}%
with $\mathbf{f}=(\mathbf{w}+\mathbf{e}_{1})\cdot \nabla \mathbf{w}-\Delta 
\mathbf{w}+\nabla q$,

\item[(S-\emph{iii})] Given $m\in \mathbb{N}$, there exists a constant $C_{m}
$ such that 
\begin{multline}
\Vert TNS[\mathbf{w},q];{\mathcal{C}^{m}(\overline{{B(}2h/3)}\setminus {{B(}%
h/3)})}\Vert \\
\leq C_{m}\left[ \left(1+ \Vert \mathbf{w};{\mathcal{C}^{m+2}(\overline{{%
\mathcal{A}}_{2}})}\Vert \right)\Vert \mathbf{w};{\mathcal{C}^{m+2}(\overline{{%
\mathcal{A}}_{2}})}\Vert +\Vert q;{\mathcal{C}^{m+1}(\overline{{\mathcal{A}}%
_{2}})}/\mathbb{R}\Vert \right] ~.  \label{eq_controlTNS}
\end{multline}

\end{enumerate}

{Applying this construction to any weak solution of Problem~1 yields a
solution of Problem~2 for the source term computed with $TNS.$ To prepare
the last weak-strong uniqueness argument of this section, we show that such
solutions of Problem~2 obtained by truncation satisfy a further energy
property. This is the content of the next proposition. }

\begin{proposition}
\label{thm_existencews2} {Given $\varepsilon $ such that $S_{\varepsilon
}\subset \subset B(h/4)$} and a weak solution $\mathbf{u}$ of {Problem 1 for}
$S_{\varepsilon }$ with associated pressure $p$, the vector-field $\tilde{%
\mathbf{u}}=\mathbf{T}_{v}[\mathbf{u}]$ satisfies

\begin{enumerate}
\item[(i)] $\tilde{\mathbf{u}} \in D$,

\item[(ii)] for all $\mathbf{w}\in \mathcal{D},$ there holds: 
\begin{equation}
\int_{\Omega _{+}}\nabla \tilde{\mathbf{u}}:\nabla \mathbf{w}\,d\text{$%
\mathbf{x}$}+\int_{\Omega _{+}}\left[ (\tilde{\mathbf{u}}+\mathbf{e}%
_{1})\cdot \nabla \tilde{\mathbf{u}}\right] \cdot \mathbf{w}\,d\text{$%
\mathbf{x}$}=\int_{\Omega _{+}}\widetilde{\mathbf{f}}\cdot \mathbf{w}\,d%
\text{$\mathbf{x}$}~,  \label{eq_wfv}
\end{equation}%
and 
\begin{equation}
\int_{\Omega _{+}}|\nabla \tilde{\mathbf{u}}|^{2}\,d\text{$\mathbf{x}$}\leq
\int_{\Omega _{+}}\widetilde{\mathbf{f}}\cdot \tilde{\mathbf{u}}\,d\text{$%
\mathbf{x}$}~,  \label{eq_NRJv}
\end{equation}%
with $\widetilde{\mathbf{f}}=TNS[\mathbf{u},p]$.
\end{enumerate}
\end{proposition}

As for the case of Problem$~$1, we emphasize 
that $\widetilde{\mathbf{f}}$ and the test-functions $\mathbf{w}$ have compact support so
that the integrals in (\ref{eq_wfv}) and (\ref{eq_NRJv}) are well-defined. 
A velocity-field $\tilde{\mathbf{u}}$ satisfying $(i)$ and $(ii)$ for a given $%
\widetilde{\mathbf{f}}\in \mathcal{C}_{c}^{\infty }(\Omega _{+})$ is called
a \textbf{weak solution} for Problem~2 with source term $\widetilde{\mathbf{f%
}}$.

\medskip 

\begin{proof}
First {we recall that ellipticity estimates for the Stokes system imply that}
any weak solution $\mathbf{u}$ for $S_{\varepsilon }$ with associated
pressure $p$ satisfies
\begin{equation*}
(\mathbf{u},p)\in \left( D\cap \mathcal{C}^{\infty }(\overline{\Omega _{+}}%
\setminus {B(}h/4))\right) \times \mathcal{C}^{\infty }(\overline{\Omega _{+}%
}\setminus {B(}h/4))~.
\end{equation*}%
Hence $\tilde{\mathbf{u}}=\mathbf{T}_{v}[\mathbf{u}]$, $\tilde{p}=T_{\pi }[p]$ and $%
TNS[\mathbf{u},p]$ are well-defined. Moreover $(\mathbf{u},p)$ is a
classical solution of the Navier Stokes equations outside $S_{\varepsilon }$
and in particular in $\overline{\Omega_{+}}\setminus {B(h/4)}$.

In order to show that $\tilde{\mathbf{u}}$ is a weak solution of Problem$~$2
we first use {(T-\emph{i})} to conclude that $\tilde{\mathbf{u}}\in D$.
Then, since $(\mathbf{u},p)$ is a classical solution to the Navier Stokes
equations in $\Omega _{+}\setminus \overline{B(h/4)}$, the second point (S-$ii$)
implies that we have 
\begin{equation*}
(\tilde{\mathbf{u}}+\mathbf{e}_{1})\cdot \nabla \tilde{\mathbf{u}}-\Delta 
\tilde{\mathbf{u}}+\nabla \tilde{p}=TNS[\mathbf{u},p]
\end{equation*}%
in $\Omega _{+}$. If we multiply this equality by $\mathbf{w}\in \mathcal{D}$
and integrate by parts we obtain (\ref{eq_wfv}) for $\tilde{\mathbf{u}}$,
with ${\widetilde{\mathbf{f}}}=TNS[\mathbf{u},p]$.

The main difficulty of the proof is to obtain the energy estimate (\ref%
{eq_NRJv}) for $\tilde{\mathbf{u}}$. For this purpose, we multiply the
Navier Stokes equations satisfied by $(\tilde{\mathbf{u}},\tilde{p})$ on ${B(%
}5h/6)$ by $\mathbf{\tilde{u}}$. Integrating by parts yields 
\begin{equation}
\int_{{{B(}5h/6)}}|\nabla \tilde{\mathbf{u}}|^{2}~d\mathbf{x}=\int_{{{B(}%
5h/6)}}\widetilde{\mathbf{f}}\cdot \tilde{\mathbf{u}}~d\mathbf{x}+\int_{\partial {{B(}%
5h/6)}}\left[ T(\tilde{\mathbf{u}},\tilde{p})\mathbf{n}\cdot \tilde{\mathbf{u%
}}+\dfrac{|\tilde{\mathbf{u}}|^{2}}{2}(\tilde{\mathbf{u}}+\mathbf{e}_{1})\cdot \mathbf{n}%
\right] ~d\text{$\sigma $}~.  \label{eq_demoNRJ1}
\end{equation}%
Next, multiplying the Navier Stokes equation satisfied by $(\mathbf{u},p)$
on ${{B(}5h/6)}\setminus \overline{S_{\varepsilon }}$ by $\mathbf{u}$ and
integrating by parts gives 
\begin{equation}
\int_{{{B(}5h/6)}}|\nabla \mathbf{u}|^{2}~d\mathbf{x}=\mathbf{\Sigma }\cdot 
\mathbf{e}_{1}+\int_{\partial {{B(}5h/6)}}\left[ T(\mathbf{u},{p})\mathbf{n}%
\cdot {\mathbf{u}}+\dfrac{|{\mathbf{u}}|^{2}}{2}(\mathbf{u}+\mathbf{e}%
_{1})\cdot \mathbf{n}\right] ~d\text{$\sigma ~,$}  \label{eq_toto1}
\end{equation}%
with $\mathbf{\Sigma }$ the associated force applied on $S_{\varepsilon }$.
By definition, we have 
\begin{equation}
\int_{\Omega _{+}}|\nabla \mathbf{u}|^{2}~d\mathbf{x}\leq \mathbf{\Sigma }%
\cdot \mathbf{e}_{1}~.  \label{eq_toto2}
\end{equation}%
Subtracting (\ref{eq_toto1}) from \eqref{eq_toto2} yields%
\begin{equation}
\int_{\Omega _{+}\setminus \overline{{B(}5h/6)}}|\nabla \mathbf{u}|^{2}~d%
\mathbf{x}\leq -\int_{\partial {{B(}5h/6)}}\left[ T(\mathbf{u},p)\mathbf{n}%
\cdot {\mathbf{u}}+\dfrac{|{\mathbf{u}}|^{2}}{2}(\mathbf{u}+\mathbf{e}%
_{1})\cdot \mathbf{n}\right] ~d\sigma ~.  \label{eq_demoNRJ2}
\end{equation}%
Since outside ${B(}2h/3)$ we have by construction that $\mathbf{u}=\tilde{%
\mathbf{u}}$ and $p=\tilde{p}$, we get by combining (\ref{eq_demoNRJ1}) and (%
\ref{eq_demoNRJ2}) 
\begin{equation*}
\int_{\Omega _{+}}|\nabla \tilde{\mathbf{u}}|^{2}~d\mathbf{x}\leq
\int_{\Omega _{+}}TNS[\mathbf{u},p]\cdot \tilde{\mathbf{u}}~d\mathbf{x}~.
\end{equation*}%
This completes the proof.
\end{proof}

\subsection{Existence of $\protect\alpha $-solutions}

{The second step of the proof of }{\bf Theorem \ref{mainexistencetheorem}}{\ is to
construct an $\alpha $-solution for Problem~2 with the source term $%
\widetilde{\mathbf{f}}_{\varepsilon }$ obtained 
by truncation of a weak solutions $\mathbf{u}_{\varepsilon }$%
. } To keep this paper self-contained, we recall the definition and the main
properties of $\alpha $-solutions. See \cite{Hillairet&Wittwer09}, for
details.

\begin{definition}
We define for fixed $\alpha $, $r\geq 0$ the function $\mu _{\alpha
,r}\colon \mathbb{R\times }[1,\infty ) \to (0,\infty)$ by 
\begin{equation}
\mu _{\alpha ,r}(k,t)=\frac{1}{1+\left( \left\vert k\right\vert t^{r}\right)
^{\alpha }}~.  \label{mu}
\end{equation}%
We define, for fixed $\alpha \geq 0$, and $p$, $q$ $\geq 0$, $\mathcal{B}%
_{\alpha ,p,q}$ to be the Banach space of functions $\hat{f}\in \mathcal{C}(%
\mathbb{R}_{0}\times \lbrack 1,\infty ),\mathbb{C})$, $\mathbb{R}_{0}=%
\mathbb{R}\setminus \{0\}$, for which the norm 
\begin{equation*}
\left\Vert \hat{f};~\mathcal{B}_{\alpha ,p,q}\right\Vert =\sup_{t\geq
1}\sup_{k\in \mathbb{R}_{0}}\frac{\left\vert \hat{f}(k,t)\right\vert }{\frac{%
1}{t^{p}}\mu _{\alpha ,1}(k,t)+\frac{1}{t^{q}}\mu _{\alpha ,2}(k,t)}
\end{equation*}%
is finite. Furthermore, we set 
$\mathcal{U}_{\alpha }=\mathcal{B}_{\alpha,\frac{5}{2},1} \times \mathcal{B}_{\alpha ,\frac{1}{2},0}\times \mathcal{B}_{\alpha ,\frac{1}{2},1}$.
\end{definition}

{ Formally, it is possible to compute the velocity-field $\mathbf{u} = (u,v),$ of a solution $(\mathbf u,p)$ to Problem 2 with source term $\widetilde{\mathbf{f}} := (F_1,F_2)$,  
as the inverse fourier transform, with respect to $x,$ of a pair $(\hat{u},\hat{v})$ : 
$$
u(x,y) = \int_{\mathbb R} e^{ikx} \hat{u}(k,y) \text{d$k$},
\qquad
v(x,y) = \int_{\mathbb R} e^{ikx} \hat{v}(k,y) \text{d$k$},
\quad
\forall \, (x,y) \in \Omega_+,
$$ 
the pair $(\hat{u},\hat{v})$ satisfying:
$$
\hat{u}(k,y) = - \hat{\eta}(k,y) + \hat{\phi}(k,y) \qquad \hat{v}(k,y) = \hat{\omega}(k,y) + \hat{\psi}(k,y), \quad \forall \, (k,y) \in \mathbb R \times (1,\infty),
$$ 
with $(\hat{\omega},\hat{\eta},\hat{\phi},\hat{\psi})$ a solution to 
\begin{eqnarray}
\partial_y \hat{\omega} &=& - i k \hat{\eta} + \hat{Q}_1 ,  \label{eq_Fourier1}\\
\partial_y \hat{\eta} & = & (ik+1) \hat{\omega} + \hat{Q}_0,\\
\partial_y \hat{\psi} & = & ik\hat{\phi} - \hat{Q}_1, \\
\partial_y \hat{\phi} & = & -ik \hat{\psi} + \hat{Q}_0.
\end{eqnarray}
The source terms $(\hat{Q}_0,\hat{Q}_1)$ is computed as follows :
\begin{eqnarray}
\hat{Q}_0 &=& \dfrac{1}{2\pi} (\hat{u} * \hat{\omega}) + \hat{F}_2 , \\
\hat{Q}_1 &=& \dfrac{1}{2\pi} (\hat{v} * \hat{\omega}) - \hat{F}_1. \label{eq_Fourier2}
\end{eqnarray}
Here $\hat{F}_1$ and $\hat{F}_2$ stand for the fourier transform, with respect to $x$, of $F_1$ and $F_2$ respectively.
When the solution $(\hat{\omega},\hat{u},\hat{v})$ given by the solution of \eqref{eq_Fourier1}--\eqref{eq_Fourier2}
satisfies $(\hat{\omega},\hat{u},\hat{v}) \in \mathcal{U}_{\alpha}$ with $\alpha >3,$ the velocity-field $\mathbf{u} =  (u,v)$
constructed this way is a weak solution to Problem 2 in the sense of {\bf Proposition \ref{thm_existencews2}}.
}
\medskip

In \cite{Hillairet&Wittwer09} the following existence theorem  is proved:

\begin{theorem}
\label{thm_existenceas} Let $\alpha >3$, $\mathbf{f}\in \mathcal{C}%
_{c}^{\infty }(\Omega _{+})$, and let $\mathbf{\widehat{f}}$ be the Fourier
transform with respect to $x$ of $\mathbf{f}$. If $\Vert \mathbf{\widehat{f}}%
;\mathcal{W}_{\alpha }\Vert $ is sufficiently small, then there exists 
an $\alpha $-solution $\mathbf{\bar{u}}$ being the inverse Fourier transform
(with respect to $x$) of $\mathbf{\hat{u}}\in \mathcal{U}_{\alpha }$, with $%
\mathbf{\hat{u}}$ satisfying $\Vert \mathbf{\hat{u}};\mathcal{U}_{\alpha
}\Vert \leq C_{\alpha }\Vert \mathbf{\widehat{f}};\mathcal{W}_{\alpha }\Vert 
$, for some constant $C_{\alpha }$ depending only on the choice of $\alpha $.
\end{theorem}
The $\alpha $-solution $\bar{\mathbf{u}}$ satisfies:

\begin{enumerate}
\item $\bar{\mathbf{u}} \in H^1_0(\Omega_+),$

\item there exists a constant $C$ such that: 
\begin{equation*}
\Vert \bar{\mathbf{u}};H_{0}^{1}(\Omega _{+})\Vert \leq C\Vert \mathbf{\hat{u%
}};\mathcal{U}_{\alpha }\Vert ~,\quad \text{ and}\quad |\bar{\mathbf{u}}%
(x,y)|\leq C\dfrac{\Vert \mathbf{\hat{u}};\mathcal{U}_{\alpha }\Vert }{%
y^{3/2}}~,\quad \forall \,(x,y)\in \Omega _{+}~.
\end{equation*}
\end{enumerate}

We now show that when the obstacle size is small, the function $\widetilde{\mathbf{f}}%
_{\varepsilon }=TNS[\mathbf{u}_{\varepsilon },p_{\varepsilon }]$ satisfies
the condition of {\bf Theorem \ref{thm_existenceas}}. This reads:

\begin{lemma}
\label{lem_petitesse}Given $\alpha > 3,$ there exists $\varepsilon
_{\alpha }>0$ such that, for all $\varepsilon <\varepsilon _{\alpha }$ {the}
weak solution $\mathbf{u}_{\varepsilon }$ with associated pressure $%
p_{\varepsilon }$ is such that {Problem~2} with source term $\widetilde{\mathbf{f}}%
_{\varepsilon }=TNS\left[ \mathbf{u}_{\varepsilon },p_{\varepsilon }\right] $
admits an $\alpha $-solution $\mathbf{\bar{u}}_{\varepsilon }$. Moreover,
there exists $C_{\alpha }<\infty $ depending only on $\alpha $ such that the 
$\alpha $-solution satisfies $\Vert \mathbf{\hat{u}}_{\varepsilon };\mathcal{%
U}_{\alpha }\Vert \leq C_{\alpha }\Vert \mathbf{u}_{\varepsilon };D\Vert $,
where $\mathbf{\hat{u}}_{\varepsilon }$ is the Fourier transform of $\mathbf{%
\bar{u}}_{\varepsilon }$ with respect to $x$.
\end{lemma}

\begin{proof}
{First, let $\eta _{0}$ be a sufficiently small parameter to be fixed later
on { and denote by $m$ the integer part of  $\alpha+1$}. Applying \textbf{Theorem \ref{thm_cvg0main}}, there exists $\varepsilon
_{\alpha ,\eta }$ such that for all $\varepsilon <\varepsilon _{\alpha ,\eta
}$ the weak solution $\mathbf{u}_{\varepsilon }$ with associated pressure $%
p_{\varepsilon }$ satisfy : 
\begin{equation*}
\Vert \mathbf{u}_{\varepsilon };\mathcal{C}^{m+2}(\overline{\mathcal{A}%
_{2}})\Vert +\Vert p_{\varepsilon };\mathcal{C}^{m +1}(\overline{%
\mathcal{A}_{2}})/\mathbb{R}\Vert \leq C_{m }\Vert \mathbf{u}%
_{\varepsilon };D\Vert \leq C_{\alpha }\eta _{0}~.
\end{equation*}%
As a consequence, the source-term $\widetilde{\mathbf{f}}_{\varepsilon
}:=TNS[\mathbf{u}_{\varepsilon },p_{\varepsilon }]$ obtained after
truncation satisfies (see (S-\emph{i}) and (S-\emph{iii})): }

\begin{itemize}
\item $\widetilde{\mathbf{f}}_{\varepsilon}$ has compact support in $%
\overline{B(3h/4)} \setminus B(h/3)$

\item $\Vert \widetilde{\mathbf{f}}_{\varepsilon };{ \mathcal{C}^{m +2}(%
\overline{\mathcal{A}_{2}}})\Vert \leq K_{\alpha }\Vert \mathbf{u}%
_{\varepsilon };\;D\Vert \leq K_{\alpha }\eta _{0}$
\end{itemize}

Denoting by ${f}$ any component of $\widetilde{\mathbf{f}}_{\varepsilon }$
we apply then the following classical computation. The function $f\in $ $%
\mathcal{C}_{c}^{\infty }(\overline{\Omega _{+}})$ has support in ${B(}2h/3)$. Hence,
the Fourier transform $\hat{f}$ of $f$ is well-defined and continuous on $%
\Omega _{+}$. Moreover we have, for $y\geq 1$ and $k\in \mathbb{R}$, 
\begin{equation*}
\hat{f}(k,y)=\int_{-2h/3}^{2h/3}e^{ikx}f(x,y)~dx~.
\end{equation*}%
Integration by parts implies the existence of a constant $C$ such that, for $y\geq 1$ and $k\in \mathbb{R}_0$, 
\begin{equation*}
\left\vert \hat{f}(k,y)\right\vert \leq C \Vert f;\mathcal{C}^{0}(\Omega
_{+})\Vert ~,\quad \text{and}\quad \left\vert \hat{f}(k,y)\right\vert \leq C%
\frac{\Vert f;\mathcal{C}^{m}(\Omega _{+})\Vert }{\left\vert k\right\vert
^{m}}~.
\end{equation*}%
Using that $\hat{f}$ has compact support in $y$, we { obtain that}%
\begin{equation*}
{
\left\vert \hat{f}(k,y)\right\vert \leq C\left[  \frac{\Vert f;\mathcal{C}^{m}(\Omega _{+})\Vert }{y^{p}\left(
1+(\left\vert k\right\vert y)^{m}\right) }+%
\frac{\Vert f;\mathcal{C}^{m}(\Omega _{+})\Vert }{y^{q}\left( 1+(\left\vert
k\right\vert y^{2})^{m}\right) }\right] ~,
}
\end{equation*}%
for arbitrary $m\in \mathbb{N}.$ In particular, there holds: 
\begin{equation}
\Vert \hat{f};\mathcal{B}_{\alpha ,p,q}\Vert \leq K_{p,q}^{\alpha }\Vert f;%
\mathcal{C}^{\alpha }(\Omega _{+})\Vert ~.  \label{eq_Fourier}
\end{equation}%
{Keeping the previous notations for the Fourier transform, we have $\Vert 
\widehat{\mathbf{f}}_{\varepsilon };\mathcal{W}_{\alpha }\Vert \leq
K_{\alpha }\Vert \mathbf{u}_{\varepsilon };\;D\Vert \leq K_{\alpha }\eta _{0}
$. Finally, for $\eta _{0}$ sufficiently small we apply }{\bf Theorem \ref%
{thm_existenceas}}.{\ This yields an $\alpha $-solution $\bar{\mathbf{u}}%
_{\varepsilon }$ for Problem~2 with source term $\widetilde{\mathbf{f}}%
_{\varepsilon }.$ Furthermore, this solution satisfies: 
\begin{equation*}
\Vert \mathbf{\hat{u}}_{\varepsilon };\mathcal{U}_{\alpha }\Vert \leq \tilde{%
C}_{\alpha }\Vert \widehat{\mathbf{f}}_{\varepsilon };\mathcal{W}_{\alpha
}\Vert \leq \tilde{C}_{\alpha }K_{\alpha }\Vert \mathbf{u}_{\varepsilon
};D\Vert ~.
\end{equation*}%
This completes the proof. }
\end{proof}

\subsection{Weak-strong uniqueness {of solution for Problem~2}}

{So far, we have shown that a weak solution $\mathbf{u}$ of Problem~1 for $%
S_{\varepsilon }$ with associated pressure $p$ provides a weak solution $%
\mathbf{\tilde{u}}$ of Problem~2 for source term $\widetilde{\mathbf{f}}=TNS\left[ 
\mathbf{u},p\right] $ by truncation. We have also shown that, for small obstacles, we can construct an $\alpha $-solution $\mathbf{%
\bar{u}}_{\varepsilon }$ for source terms $\widetilde{\mathbf{f}}_{\varepsilon }$
obtained after truncation of $\mathbf{u}_{\varepsilon }$. In this section,
we prove: }

\begin{theorem}
\label{thm_weakstrongunique} Given $\alpha >3$, there exists $\eta _{\alpha
}>0$ such that, given an $\alpha $-solution 
$\mathbf{\bar{u}}$ 
for source-term $\widetilde{\mathbf{f}}\in \mathcal{C}_{c}^{\infty }(\Omega _{+})$ such
that $\Vert \mathbf{\hat{u}};{\mathcal{U}_{\alpha }}\Vert <\eta _{\alpha },$
any weak solution $\mathbf{\tilde{u}}$ of Problem 2 with source term $%
\widetilde{\mathbf{f}}$ coincides with $\mathbf{\bar{u}}$.
\end{theorem}

{Consequently, choosing $\alpha =4,$ for instance, and a 
sufficiently small obstacle, {we}} have, by {\bf Lemma \ref{lem_petitesse}} {{that $\Vert {%
\mathbf{\hat{u}}}_{\varepsilon };\mathcal{U}_{\alpha }\Vert \leq \eta
_{\alpha }.$ Hence, we 
can apply this theorem to $\mathbf{\tilde{u}}_{\varepsilon }:=\mathbf{T}_{v}[%
\mathbf{u}_{\varepsilon }]$.} This yields that $\mathbf{\tilde{u}}%
_{\varepsilon }$ coincides with $\mathbf{\bar{u}}_{\varepsilon }$}. Since by
construction the weak solution $\mathbf{\tilde{u}}_{\varepsilon }$ coincides
with $\mathbf{u}_{\varepsilon }$ outside a compact set, $\mathbf{u}%
_{\varepsilon }$ also coincides with the $\alpha $-solution $\mathbf{\bar{u}}%
_{\varepsilon }$ outside a compact set {and inherits its asymptotic
properties}. {Thus, this weak-strong uniqueness result ends the proof of }%
{\bf Theorem \ref{mainexistencetheorem}}.

{\bf Theorem \ref{thm_weakstrongunique}} is a generalization of \cite[Theorem 8]%
{Hillairet&Wittwer09}, where it was shown that for small $\widetilde{\mathbf{f}}$ any
weak solution $\mathbf{\tilde{u}}\in H_{0}^{1}(\Omega _{+})$ of Problem$~$1,
is an $\alpha $-solution. This theorem was not general enough for the
present purposes because weak solutions that are obtained by truncation
merely satisfy $\mathbf{\tilde{u}}\in D$. The remainder of this section is
devoted to this new uniqueness proof.

\subsubsection{{Sketch of proof for} Theorem \protect\ref%
{thm_weakstrongunique}}

We set $\alpha >3$ and fix $\mathbf{\bar{u}}$ an $\alpha $-solution with a
source-term $\widetilde{\mathbf{f}}$. It has been shown in \cite[Section 3]%
{Hillairet&Wittwer09}, that such an $\alpha $-solution is also a weak
solution of Problem 2 for $\widetilde{\mathbf{f}}$. Hence, we have (\ref{eq_wfv}) and (%
\ref{eq_NRJv}) for $\tilde{\mathbf{u}}$ and $\mathbf{\bar{u}}.$ To estimate $\tilde{\mathbf{%
u}}-\mathbf{\bar{u}}$, we use the $D-$norm 
\begin{equation*}
\Vert \tilde{\mathbf{u}}-\mathbf{\bar{u}};D\Vert ^{2}=\Vert \tilde{\mathbf{u}};D\Vert
^{2}+\Vert \mathbf{\bar{u}};D\Vert ^{2}-2\int_{\Omega _{+}}\nabla \tilde{\mathbf{u}}%
:\nabla \mathbf{\bar{u}}~d\mathbf{x}~,
\end{equation*}%
where, applying (\ref{eq_NRJv}) : 
\begin{equation*}
\Vert \tilde{\mathbf{u}};D\Vert ^{2}\leq \int_{\Omega _{+}}\widetilde{\mathbf{f}}\cdot \tilde{\mathbf{u}}%
~d\mathbf{x}~,\qquad \text{and }\qquad \Vert \mathbf{\bar{u}};D\Vert
^{2}\leq \int_{\Omega _{+}}\widetilde{\mathbf{f}}\cdot \mathbf{\bar{u}}~d\mathbf{x}~.
\end{equation*}%
We now assume that
\begin{equation}
\int_{\Omega _{+}}\nabla \tilde{\mathbf{u}}:\nabla \mathbf{\bar{u}}~d\mathbf{x}%
+\int_{\Omega _{+}}(\tilde{\mathbf{u}}+\mathbf{e}_{1})\cdot \nabla \tilde{\mathbf{u}}\cdot 
\mathbf{\bar{u}}~d\mathbf{x}=\int_{\Omega _{+}}\widetilde{\mathbf{f}}\cdot \mathbf{\bar{u%
}}~d\mathbf{x}~,  \tag{H1}  \label{H1}
\end{equation}%
and that\begin{equation}
\int_{\Omega _{+}}\nabla \mathbf{\bar{u}}:\nabla {\tilde{\mathbf{u}}}~d\mathbf{x}%
+\int_{\Omega _{+}}(\mathbf{\bar{u}}+\mathbf{e}_{1})\cdot \nabla \mathbf{%
\bar{u}}\cdot {\tilde{\mathbf{u}}}~d\mathbf{x}=\int_{\Omega _{+}}\widetilde{\mathbf{f}}\cdot {%
\tilde{\mathbf{u}}}~d\mathbf{x}~.  \tag{H2}  \label{H2}
\end{equation}%
These assumptions are proved below. Combining 
\eqref{H1} and \eqref{H2} yields 
\begin{equation*}
\Vert \tilde{\mathbf{u}}-\mathbf{\bar{u}};D\Vert ^{2}\leq \int_{\Omega _{+}}(\mathbf{%
u}+\mathbf{e}_{1})\cdot \nabla \tilde{\mathbf{u}}\cdot \mathbf{\bar{u}}+\int_{\Omega
_{+}}(\mathbf{\bar{u}}+\mathbf{e}_{1})\cdot \nabla \mathbf{\bar{u}}\cdot {%
\tilde{\mathbf{u}}}~.
\end{equation*}%
Next we assume that for any $\alpha $-solution $\mathbf{\bar{u}}$ 
we have: 
\begin{equation}
\int_{\Omega _{+}}(\mathbf{v}+\mathbf{e}_{1})\cdot \nabla \mathbf{w}\cdot 
\mathbf{\bar{u}}~d\mathbf{x}=-\int_{\Omega _{+}}(\mathbf{v}+\mathbf{e}%
_{1})\cdot \nabla \mathbf{\bar{u}}\cdot \mathbf{w}~d\mathbf{x},\qquad
\forall \,(\mathbf{v,w})\in D^{2}~.  \tag{H3}  \label{H3}
\end{equation}
This assumption is also proved below. Together with the previous inequality
we get
\begin{equation*}
\begin{array}{rcl}
\Vert \tilde{\mathbf{u}}-\mathbf{\bar{u}};D\Vert ^{2} & \leq  & -\displaystyle{%
\int_{\Omega _{+}}}(\tilde{\mathbf{u}}+\mathbf{e}_{1})\cdot \nabla {\mathbf{\bar{u}}}%
\cdot \tilde{\mathbf{u}}~d\mathbf{x}+\displaystyle{\int_{\Omega _{+}}}(\mathbf{\bar{u%
}}+\mathbf{e}_{1})\cdot \nabla \mathbf{\bar{u}}\cdot {\tilde{\mathbf{u}}}~d\mathbf{x}
\\[8pt]
& \leq  & \displaystyle{\int_{\Omega _{+}}}(\mathbf{\bar{u}}-{\tilde{\mathbf{u}}}%
)\cdot \nabla \mathbf{\bar{u}}\cdot {\tilde{\mathbf{u}}}~d\mathbf{x} \\[8pt]
& \leq  & \displaystyle{\int_{\Omega _{+}}}(\mathbf{\bar{u}}-{\tilde{\mathbf{u}}}%
)\cdot \nabla \mathbf{\bar{u}}\cdot ({\tilde{\mathbf{u}}}-\mathbf{\bar{u}})~d\mathbf{%
x}~,%
\end{array}%
\end{equation*}%
{ as it yields from \eqref{H3} :
$$
 \displaystyle{\int_{\Omega _{+}}}(\mathbf{\bar{u}}-{\tilde{\mathbf{u}}}%
)\cdot \nabla \mathbf{\bar{u}}\cdot {{\mathbf{\bar u}}}~d\mathbf{x} = 0.
$$
}
Finally, we assume that
for any $\alpha $-solution $\mathbf{\bar{u}}$ 
we have:
\begin{equation}
\left\vert \int_{\Omega _{+}}\mathbf{v}\cdot \nabla \mathbf{\bar{u}\cdot w}~d%
\mathbf{x}\right\vert \leq K\Vert \mathbf{\hat{u}};\mathcal{U}_{\alpha
}\Vert ~\Vert \mathbf{v};D\Vert ~\Vert \mathbf{w};D\Vert ~,\quad \forall (%
\mathbf{v,w})\in D^{2}~,  \tag{H4}  \label{H4}
\end{equation}%
and we get 
\begin{equation*}
\begin{array}{rcl}
\Vert \tilde{\mathbf{u}}-\mathbf{\bar{u}};D\Vert ^{2} & \leq  & C\Vert {\mathbf{\hat{%
u}}};\mathcal{U}_{\alpha }\Vert ~\Vert \tilde{\mathbf{u}}-\mathbf{\bar{u}};D\Vert
^{2}~.%
\end{array}%
\end{equation*}%
For $\eta _{\alpha }$ sufficiently small we have $C\Vert {\mathbf{\hat{u}}} \ ; \ {\mathcal{U}_{\alpha }}\Vert <1/2$, 
so that $\Vert \tilde{\mathbf{u}}-\mathbf{\bar{u}};D\Vert =0$. 
This completes the proof up to the technical points
\eqref{H1}--\eqref{H4} which are proved in the following sections.

\subsubsection{Proof of (H2) and (H4)}

\label{sec_H24}

We first establish some additional conditions for the trilinear form 
\begin{equation}
\int_{\Omega _{+}}(\mathbf{u}+\mathbf{e}_{1})\cdot \nabla \mathbf{v}\cdot 
\mathbf{w}~d\mathbf{x}  \label{eq_trilinearb}
\end{equation}%
to be well defined. The main tool is the Hardy inequality for functions in $%
D $:

\begin{proposition}
\label{prop_Hardy} For all $\mathbf{w}\in D$, 
\begin{equation*}
\int_{\Omega _{+}}\dfrac{|\mathbf{w}(x,y)|^{2}}{(y-1)^{2}}~dx~dy\leq 4\Vert 
\mathbf{w};D\Vert ^{2}~.
\end{equation*}
\end{proposition}

\noindent Therefore, we have the following continuity result for the
trilinear form:

\begin{proposition}
\label{prop_contb} There exists a constant $C$ such that for all $(\mathbf{%
v,w})\in D$ and all 
\begin{equation*}
\mathbf{\bar{u}}\in H_{loc}^{1}(\Omega _{+},d\mathbf{x})\cap L^{2}(\Omega
_{+},d\mathbf{x})
\;  { \text{ such that } \; 
\| y \mathbf{u}   \ ;  \ L^{\infty }(\Omega _{+}) \| < \infty\,,
}
\end{equation*}%
and%
\begin{equation*}
\nabla \mathbf{\bar{u}}\in H_{loc}^{1}(\Omega _{+},d\mathbf{x})\cap
L^{2}(\Omega _{+},{y^{2}}d\mathbf{x})
\;  {  \text{ such that }  \;
\| y^2 \nabla \mathbf{\bar u} \ ; \ L^{\infty }(\Omega _{+} \| < \infty \,,}
\end{equation*}%
respectively, we have%
\begin{equation}
\left\vert \int_{\Omega _{+}}(\mathbf{v}+\mathbf{e}_{1})\cdot \nabla \mathbf{%
w}\cdot \mathbf{\bar{u}}~d\mathbf{x}\right\vert \leq C~\left( 1+\Vert 
\mathbf{v};D\Vert \right) ~\Vert \mathbf{w};D\Vert ~\left( \Vert \mathbf{%
\bar{u}};L^{2}(\Omega _{+})\Vert +\Vert y\mathbf{\bar{u}};L^{\infty }(\Omega
_{+})\Vert \right) ~,  \label{eq_contb1}
\end{equation}%
and 
\begin{equation}
\left\vert \int_{\Omega _{+}}(\mathbf{v}+\mathbf{e}_{1})\cdot \nabla \mathbf{%
\bar{u}}\cdot \mathbf{w}~d\mathbf{x}\right\vert \leq C~\left( 1+\Vert 
\mathbf{v};D\Vert \right) ~ { \Vert \mathbf{w}/y; L^2(\Omega_+)\Vert} ~\left( \Vert y\nabla 
\mathbf{\bar{u}};L^{2}(\Omega _{+})\Vert +\Vert y^{2}\nabla \mathbf{\bar{u}};%
{L^{\infty }(\Omega _{+})}\Vert \right) ~.  \label{eq_contb2}
\end{equation}
\end{proposition}

\begin{proof}
We denote by $I_{1}$ and $I_{2}$ the integrals in (\ref{eq_contb1}) and (\ref%
{eq_contb2}), respectively. Let $(\mathbf{\bar{u}},\mathbf{v},\mathbf{w})\in
H_{loc}^{1}(\Omega _{+})\times D^{2}$. If $\Vert \mathbf{\bar{u}}%
;L^{2}(\Omega _{+})\Vert +\Vert y\mathbf{\bar{u}};L^{\infty }(\Omega
_{+})\Vert <\infty $, we can split $I_{1}$ into two integrals 
\begin{equation}
I_{1}=\int_{\Omega _{+}}\mathbf{v}\cdot \nabla \mathbf{w}\cdot \mathbf{\bar{u%
}}~d\mathbf{x}-\int_{\Omega _{+}}\partial _{x}\mathbf{w}\cdot \mathbf{\bar{u}%
}~d\mathbf{x}~.  \label{i1}
\end{equation}%
The second integral on the right hand side of (\ref{i1}) can be bounded by
the Cauchy Schwarz inequality. For the first integral we have 
\begin{eqnarray}
\left\vert \int_{\Omega _{+}}\mathbf{v}\cdot \nabla \mathbf{w}\cdot \mathbf{%
\bar{u}}~d\mathbf{x}\right\vert  &=&\left\vert \int_{\Omega _{+}}\dfrac{%
\mathbf{v}}{y}\cdot \nabla \mathbf{w}\cdot y\mathbf{\bar{u}}~d\mathbf{x}%
\right\vert  \\
&\leq &\Vert \mathbf{v}/{y};L^{2}(\Omega _{+})\Vert ~\Vert \nabla \mathbf{w}%
;L^{2}(\Omega _{+})\Vert ~\Vert y\mathbf{\bar{u}};L^{\infty }(\Omega
_{+})\Vert ~,  \notag
\end{eqnarray}%
and we obtain the bound on $I_{1}$ by applying the Hardy inequality (note
that $y\geq y-1$). The integral $I_{2}$ is bounded similarly.
\end{proof}

\medskip

{This proposition is suitable for $\alpha $-solutions. Indeed, if $\mathbf{%
\bar{u}}=(u,v)$, is an $\alpha $-solution, for $\alpha >3$, denoting by $%
\mathbf{\hat{u}}=(\hat{u},\hat{v})$ (respectively $\mathbf{\hat{u}}_{x}=(%
\hat{u}_{x},\hat{v}_{x})$ and $\mathbf{\hat{u}}_{y}=(\hat{u}_{y},\hat{v}%
_{y}) $ ) the Fourier transform with respect to $x$ of $\mathbf{\bar{u}}$
(respectively $\partial _{x}\mathbf{\bar{u}}$ and $\partial _{y}\mathbf{\bar{%
u}}$) there holds : }

\begin{itemize}
\item $(\hat{u},\hat{v})\in {\mathcal{B}}_{\alpha ,1/2,0}\times {\mathcal{B}}%
_{\alpha ,1/2,1}$,

\item $(\hat{u}_x,\hat{v}_x )\in {\mathcal{B}}_{\alpha -1,3/2,2}\times {%
\mathcal{B}}_{\alpha -1,3/2,3}$

\item $(\hat{u}_y,\hat{v}_y)\in {\mathcal{B}}_{\alpha -1,3/2,1}\times {%
\mathcal{B}}_{\alpha -1,3/2,2}$
\end{itemize}

with :%
\begin{equation*}
\Vert {\hat{u}}_{x};{\mathcal{B}}_{\alpha -1,3/2,2}\times {\mathcal{B}%
}_{\alpha -1,3/2,3}\Vert +\Vert {\hat{u}}_{y};{\mathcal{B}}_{\alpha
-1,3/2,1}\times {\mathcal{B}}_{\alpha -1,3/2,2}\Vert \leq C_{\alpha }\Vert 
\mathbf{\hat{u}};\mathcal{U}_{\alpha }\Vert ~.
\end{equation*}%
(see \cite[pp. 685-686]{Hillairet&Wittwer09} for more details). Moreover, we have:

\begin{proposition}
\label{prop_decayB} Let $p$, $q>0$, $\alpha >1$, $s\in \lbrack 2,\infty ]$
and let $f$ be the inverse Fourier transform of $\hat{f}\in {\mathcal{B}}%
_{\alpha ,p,q}$. Then there exists a constant $C$ depending only on $\alpha $
and $s$ such that, for any $y>1$, $f(\cdot ,y)\in L^{s}(\mathbb{R})$, and%
\begin{equation*}
\Vert f(\cdot ,y);L^{s}(\mathbb{R})\Vert \leq \dfrac{C}{y^{e}}\Vert \hat{f};{%
\mathcal{B}_{\alpha ,p,q}}\Vert ~.
\end{equation*}%
where $e=\min (1-1/s+p,2(1-1/s)+q)$.
\end{proposition}

\begin{proof}
Given $\hat{f}\in \mathcal{B}_{\alpha ,p,q}$, and $y>1$, the function $f(x,y)
$ is the inverse Fourier transform with respect to $k$ of the function $\hat{%
f}(k,y)$, which is continuous on $\mathbb{R}_{0}$ and satisfies for $k\in 
\mathbb{R}_{0}$ 
\begin{equation*}
|\hat{f}(k,y)|\leq \left( \dfrac{1}{y^{p}(1+(|k|y)^{\alpha })}+\dfrac{1}{%
y^{q}(1+(|k|y^{2})^{\alpha }}\right) \Vert \hat{f};{\mathcal{B}_{\alpha ,p,q}}%
\Vert ~.
\end{equation*}%
Therefore $\hat{f}(\cdot ,y)\in L^{r}(\mathbb{R})$ for all $r\in \lbrack 1,2]
$ so that $f\in L^{s}(\mathbb{R})$ for all $s\in \lbrack 2,\infty ]$.
Moreover, given $s\geq 2$ and $r\leq 2$ the conjugate exponent,\textit{\ i.e.%
} $\frac{1}{r}+\frac{1}{s}=1$, there exists a constant $C_{s}$, such that%
\begin{equation}
\Vert f(\cdot ,y);L^{s}(\mathbb{R})\Vert \leq C_{s}\Vert \hat{f}(\cdot
,y);L^{r}(\mathbb{R})\Vert ~.  \label{ftn}
\end{equation}%
By a scaling argument, we have%
\begin{equation*}
\int_{-\infty }^{\infty }\dfrac{1}{(1+|k|y)^{r\alpha }}d\text{$k$}\leq 
\dfrac{C_{r,\alpha }^{1}}{y}~,\quad \text{and}\quad \int_{-\infty }^{\infty }%
\dfrac{1}{(1+|k|y^{2})^{r\alpha }}d\text{$k$}\leq \dfrac{C_{r,\alpha }^{2}}{%
y^{2}}~,
\end{equation*}%
which, together with (\ref{ftn}) gives 
\begin{equation*}
\Vert f(\cdot ,y);L^{s}(\mathbb{R})\Vert \leq C_{\alpha ,r}\left( \dfrac{1}{%
y^{p+\frac{1}{r}}}+\dfrac{1}{y^{q+\frac{2}{r}}}\right) \Vert \hat{f};{%
\mathcal{B}_{\alpha ,p,q}}\Vert ~,
\end{equation*}%
as required.
\end{proof}

{\bf Proposition \ref{prop_decayB}} implies that, if $\hat{f}\in \mathcal{B}%
_{\alpha ,p,q}$ for $p>0$, $q>0$ and $\alpha >1$, then $f\in L^{2}(\Omega
_{+})$ and we have, for all $(x,y)\in \Omega _{+}$, 
\begin{equation*}
|f(x,y)|\leq \dfrac{C}{y^{\min (p+1,q+2)}}\Vert \hat{f};{\mathcal{B}_{\alpha
,p,q}}\Vert ~.
\end{equation*}%
In particular, for an $\alpha $-solution $\mathbf{\bar{u}}=(u,v)$, we have $%
\mathbf{\bar{u}}\in L^{2}(\Omega _{+})$, $y\nabla \mathbf{\bar{u}\in }%
L^{2}(\Omega _{+})$, and according to the remark before the proposition,
there holds : {%
\begin{equation}
|u(x,y)|+|v(x,y)|\leq \dfrac{C}{y^{3/2}}\Vert \mathbf{\hat{u}};\mathcal{U}%
_{\alpha }\Vert ~,  \label{eq_decayuv}
\end{equation}%
\begin{equation}
|\nabla u(x,y)|+|\nabla v(x,y)|\leq \dfrac{C}{y^{5/2}}\Vert \mathbf{\hat{u}};%
\mathcal{U}_{\alpha }\Vert ~,  \label{eq_decaydudv}
\end{equation}%
} for all $(x,y)\in \Omega _{+}$. Consequently, we can use the bound (\ref%
{eq_contb2}) for $\mathbf{w}\in \mathcal{D}$, and we find that{%
\begin{equation*}
\left\vert \int_{\Omega _{+}}(\mathbf{\bar{u}}+\mathbf{e}_{1})\cdot \nabla 
\mathbf{\bar{u}}\cdot \mathbf{w}~d\mathbf{x}\right\vert \leq C~\left( \Vert 
\mathbf{\hat{u}};\mathcal{U}_{\alpha }\Vert +1\right) ~\Vert \mathbf{\hat{u}}%
;\mathcal{U}_{\alpha }\Vert ~\Vert \mathbf{w};D\Vert ~.
\end{equation*}%
} This implies that for 
the $\alpha $-solution $\mathbf{\bar{u}}$ the linear form $L_{\mathbf{\bar{u}}}$,%
\begin{equation*}
L_{\mathbf{\bar{u}}}[\mathbf{w]=}\int_{\Omega _{+}}(\mathbf{\bar{u}}+\mathbf{%
e}_{1})\cdot \nabla \mathbf{\bar{u}}\cdot \mathbf{w}~d\mathbf{x}
\end{equation*}%
is continuous on $D$, so that the weak formulation (\ref{eq_wfv}) for $%
\mathbf{\bar{u}}$ can be extended to $D$. This completes the proof of assumption \eqref{H2}.

\bigskip

\noindent With similar arguments we obtain, that for arbitrary $(\mathbf{v,w}%
)\in D^{2}$%
\begin{eqnarray*}
\left\vert \int_{\Omega _{+}}\mathbf{v}\cdot \nabla \mathbf{\bar{u}\cdot w}~d%
\mathbf{x}\right\vert  &\leq &K~\Vert y^{2}\nabla \mathbf{\bar{u}};{%
L^{\infty }(\Omega _{+})}\Vert ~\Vert \mathbf{v};D\Vert ~\Vert \mathbf{w}%
;D\Vert  \\
&\leq &K\Vert \mathbf{\hat{u}};\mathcal{U}_{\alpha }\Vert ~\Vert \mathbf{v}%
;D\Vert ~\Vert \mathbf{w};D\Vert ~.
\end{eqnarray*}
This completes the proof of assumption \eqref{H4}.

\subsubsection{Proof of (H1)}

\label{sec_H1}

The following proposition shows that $\alpha $-solutions are approximated
by velocity-fields of compact support:

\begin{proposition}
\label{prop_approxalphasol} Let $\alpha >3$ and let $\mathbf{\bar{u}}:=(u,v)$
be an $\alpha $-solution. Then there exists a sequence $(\mathbf{\bar u}_n)_{n\in \mathbb{N}}\in D^{\mathbb{N}}$, such that 
for any $n \in \mathbb N$
\begin{enumerate}
\item[$i)$] $\mathbf{\bar u}_n \in \mathcal{C}^{\infty}(\Omega_+)$
\item[$ii)$] $\mathbf{\bar u}_n= \mathbf{\bar{u}}$ in $B((0,0),n)\cap \Omega _{+}$ and $\mathbf{\bar u}_n =0$ outside $B((0,0),2n)\cap \Omega_+.$
\item[$iii)$] There exists a constant $C(\mathbf{\bar{u}})$ such that, for all $n\in 
\mathbb{N}$,
\end{enumerate}
\begin{equation}
\Vert \mathbf{\bar{u}}-\mathbf{\bar u}_n;L^{\infty }{(\Omega _{+})}\Vert +\Vert
y\nabla (\mathbf{\bar{u}}-\mathbf{\bar u}_n);L^{2}{(\Omega _{+})}\Vert +\Vert y(%
\mathbf{\bar{u}}-\mathbf{\bar u}_n);L^{\infty }(\Omega _{+})\Vert +\Vert
y^{2}\nabla (\mathbf{\bar{u}}-\mathbf{\bar u}_n);L^{\infty }(\Omega _{+})\Vert
\leq C(\mathbf{\bar{u}})~.  \notag
\end{equation}
\end{proposition}

\begin{proof}
Let $\alpha >3$ and $\mathbf{\bar{u}}=(u,v)$ be an $\alpha $-solution and
let $\psi =\Pi \lbrack \mathbf{\bar{u}}]$ be the corresponding stream-function.
{\bf Proposition \ref{prop_decayB}} implies that $\mathbf{\bar{u}}\in
H_{0}^{1}(\Omega _{+})$ and we have the bounds (\ref{eq_decayuv}), (\ref%
{eq_decaydudv}). Therefore we have not only that $\psi \in \mathcal{C}%
^{1}(\Omega _{+})$ and that $\nabla ^{\bot }\psi (x,y)=\mathbf{\bar{u}}(x,y)$%
, but also that $\psi \in L^{\infty }(\Omega _{+})$, and that, for all $%
(x,y)\in \Omega _{+}$,%
\begin{equation}
|\psi (x,y)|\leq C\Vert \hat{u};\mathcal{B}_{\alpha ,1/2,0}\Vert ~,
\label{decaypsi}
\end{equation}%
with the previous notations, and, since $\mathbf{\bar{u}}$ is smooth in $%
\Omega _{+}$, the stream-function $\psi $ is also smooth in $\Omega _{+}$. Let {%
\begin{equation*}
\zeta _{n}(x,y)=\zeta \left( \dfrac{|(x,y)|}{n}-1\right) ~,\quad \text{and}%
\quad \mathbf{\bar u}_n=\nabla ^{\bot }\left[ \zeta _{n}\mathbf{\bar{u}}\right]
~.
\end{equation*}%
} Then, $\mathbf{\bar u}_n\in {D} $ and it satisfies \emph{i)} and \emph{ii)}
for any $n\in \mathbb{N}$, and 
\begin{equation*}
\mathbf{\bar u}_n-\mathbf{\bar{u}}=(\zeta _{n}-1)\mathbf{\bar{u}}+\psi \nabla
^{\bot }\zeta _{n}~.
\end{equation*}%
Using a scaling argument, one shows that $\Vert \nabla \zeta
_{n};L^{2}(\Omega _{+})\Vert $ is uniformly bounded for $n\in \mathbb{N}$,
and therefore we have the uniform bound, 
\begin{equation*}
\Vert \mathbf{\bar u}_n-\mathbf{\bar{u}};{L^{2}(\Omega _{+})}\Vert \leq \Vert 
\mathbf{\bar{u}};{L^{2}(\Omega _{+})}\Vert +C_{\zeta }\Vert \psi ;{L^{\infty
}(\Omega _{+})}\Vert ~.
\end{equation*}%
Similarly, $\Vert y\nabla \zeta _{n};L^{\infty }(\Omega _{+})\Vert $ is
uniformly bounded for $n\in \mathbb{N}$, and as a consequence we have for
all $\,(x,y)\in \Omega _{+}$%
\begin{equation}
|y(\mathbf{\bar u}_n-\mathbf{\bar{u}})(x,y)|\leq |y\mathbf{\bar{u}}%
(x,y)|+C_{\zeta }|\psi (x,y)|~.  \label{eq1}
\end{equation}%
Using \eqref{eq_decayuv} and \eqref{decaypsi}, we find that the right-hand side
in (\ref{eq1}) is uniformly bounded for $(x,y,n)\in \Omega _{+}\times 
\mathbb{N}$.

\medskip

For the derivatives of $\mathbf{\bar u}_n$ and $\mathbf{\bar{u}}$, we have%
\begin{equation}
|\nabla \mathbf{\bar u}_n(x,y)-\nabla \mathbf{\bar{u}}(x,y)|\leq |\nabla 
\mathbf{\bar{u}}(x,y)|+C~|\nabla \zeta _{n}(x,y)|~|\mathbf{\bar{u}}%
(x,y)|+|\nabla ^{2}\zeta _{n}(x,y)|~|\psi (x,y)|~.  \label{eq3}
\end{equation}%
Applying scaling techniques as above, one obtains for $(x,y,n)\in \Omega
_{+}\times \mathbb{N}$, 
\begin{equation}
\Vert y\nabla ^{2}\zeta _{n};L^{2}(\Omega _{+})\Vert \leq {C_{\zeta }}%
~,\quad \quad |y^{2}\nabla ^{2}\zeta _{n}(x,y)|\leq C_{\zeta }~.  \label{eq2}
\end{equation}%
From (\ref{eq2}) and (\ref{eq3}) and 
\eqref{eq_decayuv} we get
\begin{multline*}
\Vert y\nabla (\mathbf{\bar u}_n-\mathbf{\bar{u})};{L^{2}(\Omega _{+})}\Vert  \\%
[6pt]
\leq \Vert y\nabla \mathbf{\bar{u}};{L^{2}(\Omega _{+})}\Vert +C~\Vert
\nabla \zeta _{n};L^{2}(\Omega _{+})\Vert ~\Vert y\mathbf{\bar{u}};{%
L^{\infty }(\Omega _{+})}\Vert +\Vert y\nabla ^{2}\zeta _{n};{L^{2}(\Omega
_{+})}\Vert ~\Vert \psi ;{L^{\infty }(\Omega _{+})}\Vert ~,
\end{multline*}%
which yields a uniform bound with respect to $n$. Finally, we have, 
\begin{equation*}
y^{2}|\nabla \mathbf{\bar u}_n(x,y)-\nabla \mathbf{\bar{u}}(x,y)|\leq
|y^{2}\nabla \mathbf{\bar{u}}(x,y)|+C~|y\nabla \zeta _{n}(x,y)|~|y\mathbf{%
\bar{u}}(x,y)|+|y^{2}\nabla ^{2}\zeta _{n}(x,y)|~|\psi (x,y)|~.
\end{equation*}%
Therefore, the previous pointwise bound \eqref{eq_decaydudv} on $\nabla 
\mathbf{\bar{u}}$ implies that $\Vert y^{2}(\nabla \mathbf{\bar{u}}-\nabla \mathbf{\bar u}_n);L^{\infty
}(\Omega _{+})\Vert $ is finite and remains uniformly bounded for $n\in 
\mathbb{N}$. This completes the proof of {\bf Proposition \ref{prop_approxalphasol}}.
\end{proof}

Combining {\bf Proposition \ref{prop_contb}} and {\bf Proposition \ref%
{prop_approxalphasol}} we are now able to prove \eqref{H1}. Indeed, let $%
\mathbf{\tilde{u}}\in D$ be a weak solution for $\widetilde{\mathbf{f}}$ and let $%
\mathbf{\bar{u}}$ be the corresponding $\alpha $-solution. From {\bf Proposition %
\ref{prop_approxalphasol}} we get that 
there exists a sequence $\mathbf{\bar u}_n\in {D}^{\mathbb{N}}$ which
approximates $\mathbf{\bar{u}}$, and, since $\mathbf{\bar u}_n$ has bounded support, 
equation (\ref{eq_wfv}) is satisfied by $\mathbf{\bar u}_n$. Using the bounds satisfied by $\mathbf{\tilde{u}}$ and $%
\mathbf{\bar{u}}$ we get
\begin{equation*}
\left\vert \int_{\Omega _{+}}\nabla \mathbf{\tilde{u}}:(\nabla \mathbf{\bar{u%
}}-\nabla \mathbf{\bar u}_n)\right\vert ~d\mathbf{x}\leq C(\mathbf{\bar{u}}%
)\left( \int_{\Omega _{+}\setminus {B(}(0,0),n)}|\nabla \mathbf{\tilde{u}}%
|^{2}~d\mathbf{x}\right) ^{\frac{1}{2}}~,
\end{equation*}%
and therefore 
\begin{equation*}
\lim_{n\rightarrow \infty }\int_{\Omega _{+}}\nabla \mathbf{\tilde{u}}%
:\nabla \mathbf{\bar u}_n~d\mathbf{x}=\int_{\Omega _{+}}\nabla \mathbf{\tilde{u}%
}:\nabla \mathbf{\bar{u}}~d\mathbf{x}~.
\end{equation*}%
{
To bound the trilinear form, we now apply (\ref{eq_contb1}) and get%
\begin{eqnarray*}
&&\left\vert \int_{\Omega _{+}}(\mathbf{{v}}+\mathbf{e}_{1})\cdot
\nabla \mathbf{w}\cdot \left( \mathbf{\bar u}-\mathbf{\bar u}_n\right) ~d%
\mathbf{x}\right\vert  \\
&\leq &C~\left( 1+\Vert \mathbf{{v}};D\Vert \right) ~\Vert \nabla 
\mathbf{{w}};{L^{2}(\Omega _{+}\setminus {B(}(0,0),n))}\Vert ~\left(
\Vert \mathbf{\bar{u}}-\mathbf{\bar u}_n;L^{2}(\Omega _{+})\Vert +\Vert y\left( 
\mathbf{\bar{u}}-\mathbf{\bar u}_n\right) ;{L^{\infty }(\Omega _{+})}\Vert
\right)  \\
&\leq &C(\mathbf{\bar{u}})~\left( 1+\Vert \mathbf{v};D\Vert \right)
~\Vert \nabla \mathbf{w};{L^{2}(\Omega _{+}\setminus {B(}(0,0),n))}%
\Vert ~.
\end{eqnarray*}%
\newline
Passing to the limit in $n$, we get 
\begin{equation*}
\lim_{n\rightarrow \infty }\int_{\Omega _{+}}(\mathbf{v}+\mathbf{e}%
_{1})\cdot \nabla \mathbf{w}\cdot \mathbf{\bar u}_n~d\mathbf{x}%
=\int_{\Omega _{+}}(\mathbf{v}+\mathbf{e}_{1})\cdot \nabla \mathbf{%
w}\cdot \mathbf{\bar{u}}~d\mathbf{x}~.
\end{equation*}%
}
Finally, using that $\widetilde{\mathbf{f}}$ has compact support, we get, for $n$
sufficiently large, 
\begin{equation*}
\int_{\Omega _{+}}\widetilde{\mathbf{f}}\cdot \mathbf{\bar u}_n~d\mathbf{x}=\int_{\Omega
_{+}}\widetilde{\mathbf{f}}\cdot \mathbf{\bar{u}}~d\mathbf{x}~.
\end{equation*}%
Passing to the limit in (\ref{eq_wfv}) with $\mathbf{\bar u}_n$ we obtain (\ref%
{eq_wfv}) with $\mathbf{\bar{u}}$. This completes the proof of \eqref{H1}.

\subsubsection{Proof of (H3)}

\label{sec_H3}

With arguments similar to the 
ones in the previous subsection we show that we have for any $(\mathbf{v,w})\in D^{2}$
and $\mathbf{\bar{u}}$ an $\alpha $-solution, 
\begin{equation*}
\lim_{n\rightarrow \infty }\int_{\Omega _{+}}(\mathbf{v}+\mathbf{e}%
_{1})\cdot \nabla \mathbf{w}\cdot \mathbf{\bar u}_n~d\mathbf{x}=\int_{\Omega
_{+}}(\mathbf{v}+\mathbf{e}_{1})\cdot \nabla \mathbf{w}\cdot \mathbf{\bar{u}}%
~d\mathbf{x}~.
\end{equation*}%
From (\ref{eq_contb2}) we get 
\begin{multline*}
\left\vert \int_{\Omega _{+}}(\mathbf{v}+\mathbf{e}_{1})\cdot \nabla (%
\mathbf{\bar{u}}-\mathbf{\bar u}_n)\cdot \mathbf{w}~d\mathbf{x}\right\vert  \\
\begin{array}{l}
\leq C\left( 1+\Vert \mathbf{v};D\Vert \right) \Vert \frac{\mathbf{w}}{y};{%
L^{2}(\Omega _{+}\setminus {B(}(0,0),n))}\Vert \Big[\Vert y^{2}\nabla (%
\mathbf{\bar u}_n-\mathbf{\bar{u}});{L^{\infty }(\Omega _{+})}\Vert +\Vert
y\nabla (\mathbf{\bar u}_n-\mathbf{\bar{u}});{L^{2}(\Omega _{+})}\Vert \Big] \\%
[8pt]
\leq C(\mathbf{\bar{u}})~\left( 1+\Vert \mathbf{v};D\Vert \right) ~\Vert 
\mathbf{w}/y;{L^{2}(\Omega _{+}\setminus {B(}(0,0),n))}\Vert ~.%
\end{array}%
\end{multline*}%
The Hardy inequality implies that $\mathbf{w}/y\in L^{2}(\Omega _{+})$.
Consequently, we have the following limit, 
\begin{equation*}
\lim_{n\rightarrow \infty }\int_{\Omega _{+}}(\mathbf{v}+\mathbf{e}%
_{1})\cdot \nabla \mathbf{\bar u}_n\cdot \mathbf{w}~d\mathbf{x}=\int_{\Omega
_{+}}(\mathbf{v}+\mathbf{e}_{1})\cdot \nabla \mathbf{\bar{u}}\cdot \mathbf{w}%
~d\mathbf{x}~.
\end{equation*}%
Since the approximation $\mathbf{\bar u}_n$ has compact support, for any fixed $n\in \mathbb{N}
$, we have 
\begin{equation*}
\int_{\Omega _{+}}(\mathbf{v}+\mathbf{e}_{1})\cdot \nabla \mathbf{w}\cdot 
\mathbf{\bar u}_n~d\mathbf{x}=-\int_{\Omega _{+}}(\mathbf{v}+\mathbf{e}%
_{1})\cdot \nabla \mathbf{\bar u}_n\cdot \mathbf{w}~d\mathbf{x}~.
\end{equation*}%
Therefore, the same identity is true for $\mathbf{w}$. This proves (\ref{H3}).

\section{Uniqueness of solutions for Problem 1}

To conclude the paper, we sketch the proof that weak solutions for small obstacles are also
unique. {We note that this result is not included {\bf Theorem \ref{thm_weakstrongunique}}. First, this {previous } theorem applies
to $\mathbf{\tilde{u}}_{\varepsilon }$. Hence, it gives information on $%
\mathbf{u}_{\varepsilon }$ only far from $S_{\varepsilon }$ where $\mathbf{%
\tilde{u}}_{\varepsilon }$ coincides with $\mathbf{u}_{\varepsilon }${.}
Second, the "unique" $\alpha $-solution ${\mathbf{\bar{u}}}_{\varepsilon }$,
to which ${\mathbf{\tilde{u}}}_{\varepsilon }$ is compared, depends itself
on the source term obtained from $\mathbf{u}_{\varepsilon }$ in the
truncation procedure. However another weak solution to Problem 1 could
create another source term.} Our final result is

\begin{theorem}
There exists $\varepsilon ^{u}>0$, such that, for all $\varepsilon
<\varepsilon ^{u}$, {if $\mathbf{u}$ is a weak solution to Problem 1 for $%
S_{\varepsilon }$ then $\mathbf{u}=\mathbf{u}_{\varepsilon }$.}
\end{theorem}

\begin{proof}
{The following proof is very close to the proof of {\bf Theorem \ref{thm_weakstrongunique}}. Hence, we only sketch the
main ideas.} First, we fix $\alpha >3$ and choose ${\varepsilon _{0}^{u}}$
such that, for all $\varepsilon <\varepsilon _{0}^{u}$, any weak solution $%
\mathbf{u}_{\varepsilon }$ is equal to the $\alpha $-solution $\mathbf{\bar{u%
}}_{\varepsilon }$ outside ${B(}2h/3)$. Furthermore, there holds (see {\bf Lemma %
\ref{lem_petitesse}}): 
\begin{equation*}
\Vert \mathbf{\hat{u}}_{\varepsilon };\mathcal{U}_{\alpha }\Vert \leq
C_{\alpha }\Vert \mathbf{u}_{\varepsilon };D\Vert ~.
\end{equation*}%
for some constant $C_{\alpha }$ depending only on $\alpha $. Here, $\mathbf{%
\hat{u}}_{\varepsilon }$ stands once again for the Fourier transform of ${%
\mathbf{\bar{u}_{\varepsilon }}}$ with respect to $x$. Now, let $\varepsilon
<\varepsilon _{0}^{u}$ and let $\mathbf{u}$ be a weak solution of Problem 1
for $S_{\varepsilon }$. {Following the sketch of proof of w{\bf Theorem \ref%
{thm_weakstrongunique}}, we obtain that 
\begin{equation*}
\Vert \mathbf{u}-\mathbf{u}_{\varepsilon };D\Vert ^{2}\leq \int_{\Omega
_{+}}(\mathbf{u}-\mathbf{u}_{\varepsilon })\cdot \nabla \mathbf{u}%
_{\varepsilon }\cdot (\mathbf{u}-\mathbf{u}_{\varepsilon })~d\mathbf{x}~.
\end{equation*}%
The technicalities which arise here are analogous to (H1)--(H4), and are justified by splitting integrals 
as follows:
\begin{equation*}
I(\mathbf{v},\mathbf{w},\mathbf{z}):=\int_{\Omega _{+}}(\mathbf{v}+\mathbf{e}%
_{1})\cdot \nabla \mathbf{w}\cdot \mathbf{z}~d\mathbf{x}=I_{int}(\mathbf{v},%
\mathbf{w},\mathbf{z})+I_{ext}(\mathbf{v},\mathbf{w},\mathbf{z})~,
\end{equation*}%
where 
\begin{equation*}
I_{ext}(\mathbf{v},\mathbf{w},\mathbf{z}):=\int_{\Omega _{+}\setminus {B(}%
2h/3)}(\mathbf{v}+\mathbf{e}_{1})\cdot \nabla \mathbf{w}\cdot \mathbf{z}~d%
\mathbf{x}~,\quad I_{int}(\mathbf{v},\mathbf{w},\mathbf{z}):=\int_{{B(}2h/3)}(%
\mathbf{v}+\mathbf{e}_{1})\cdot \nabla \mathbf{w}\cdot \mathbf{z}~d\mathbf{x}%
~.
\end{equation*}%
Therefore, one  proves, as in sections {\ref{sec_H24} to \ref{sec_H3},}
suitable continuity and antisymmetric properties of the trilinear form $I$, when applied to $\mathbf{u}%
_{\varepsilon }$, and using the fact that in $I_{ext}$ the weak solution $\mathbf{u}_{\varepsilon }
$ coincides with the $\alpha $-solution $\mathbf{\bar{u}}_{\varepsilon }$.}
This yields
\begin{equation*}
\begin{array}{rl}
\Vert \mathbf{u}-\mathbf{u}_{\varepsilon };D\Vert ^{2} & \leq C\Big[\Vert 
\mathbf{\hat{u}}_{\varepsilon };\mathcal{U}_{\alpha }\Vert +\Vert \mathbf{u}%
_{\varepsilon };D\Vert \Big]\Vert \mathbf{u}-\mathbf{u}_{\varepsilon
};D\Vert ^{2} \\[8pt]
& \leq C_{\alpha }\Vert \mathbf{u}_{\varepsilon };D\Vert ~\Vert \mathbf{u}-%
\mathbf{u}_{\varepsilon };D\Vert ^{2}~.%
\end{array}%
\end{equation*}%
According to {\bf Theorem \ref{thm_existencews}}, there exists  $%
\varepsilon _{1}^{u}$ such that, if $\varepsilon <\varepsilon _{1}^{u}$, we
have $\Vert \mathbf{u}_{\varepsilon};D\Vert <1/(2C_{\alpha })$. This completes the proof.
\end{proof}

\bigskip

\paragraph{Acknowledgements.}
The authors would like to thank G. P. Galdi for suggesting the technique
of proof of {\bf Theorem \ref{thm_weakstrongunique}}, and J. Guillod for 
the careful reading of preliminary versions of the paper.


\begin{thebibliography}{10}

\bibitem{Amick91}
C.J. Amick.
\newblock On the asymptotic form of {N}avier-{S}tokes flow past a body in the
  plane.
\newblock {\em Journal of differential equations}, 91:149--167, 1991.

\bibitem{Amrouche&Bonzom09}
C. Amrouche and F. Bonzom.
\newblock Exterior {S}tokes problem in the half-space.
\newblock {\em Ann. Univ. Ferrara Sez. VII Sci. Mat.}, 55(1):37--66, 2009.

\bibitem{Babenko75}
K.~I. Babenko.
\newblock Theory of perturbations of stationary flows of a viscous
  incompressible fluid in the case of small {R}eynolds numbers.
\newblock {\em Dokl. Akad. Nauk SSSR}, 227(3):592--595, 1976.

\bibitem{Christoph}
C. Boeckl and P. Wittwer.
\newblock Decay estimates for steady solutions of the navier-stokes equations
  in two dimensions in the presence of a wall.
\newblock in preparation.

\bibitem{Finn&Smith67}
R. Finn and D.~R. Smith.
\newblock On the stationary solutions of the {N}avier-{S}tokes equations in two
  dimensions.
\newblock {\em Arch. Rational Mech. Anal.}, 25:26--39, 1967.

\bibitem{Fischer&Hsiao&Wendland86}
T.~M. Fischer, G.~C. Hsiao, and W.~L. Wendland.
\newblock On two-dimensional slow viscous flows past obstacles in a half-plane.
\newblock {\em Proc. Roy. Soc. Edinburgh Sect. A}, 104(3-4):205--215, 1986.

\bibitem{Fujita61}
H. Fujita.
\newblock On the existence and regularity of the steady-state solutions of the
  {N}avier-{S}tokes theorem.
\newblock {\em J. Fac. Sci. Univ. Tokyo Sect. I}, 9:59--102 (1961), 1961.


\bibitem{Galdi94}
{\sc Giovanni~P. Galdi}, {\em An introduction to the mathematical theory of the
  {N}avier-{S}tokes equations. {V}ol. {I}}, vol.~38 of Springer Tracts in
  Natural Philosophy, Springer-Verlag, New York, 1994.
\newblock Linearized steady problems.


\bibitem{Galdi98}
G.~P. Galdi.
\newblock {\em An introduction to the mathematical theory of the
  {N}avier-{S}tokes equations. {V}ol. {II}}, volume~39 of {\em Springer Tracts
  in Natural Philosophy}.
\newblock Springer-Verlag, New York, 1994.
\newblock Nonlinear steady problems.


\bibitem{Gilbarg&Weinberger78}
D.~Gilbarg and H.~F. Weinberger.
\newblock Asymptotic properties of steady plane solutions of the
  {N}avier-{S}tokes equations with bounded {D}irichlet integral.
\newblock {\em Ann. Scuola Norm. Sup. Pisa Cl. Sci. (4)}, 5(2):381--404, 1978.

\bibitem{Gilbarg&Trudinger01}
D. Gilbarg and N.~S. Trudinger.
\newblock {\em Elliptic partial differential equations of second order}.
\newblock Classics in Mathematics. Springer-Verlag, Berlin, 2001.
\newblock Reprint of the 1998 edition.

\bibitem{Hillairet07}
M. Hillairet.
\newblock Chute stationnaire d'un solide dans un fluide visqueux incompressible
  le long d'un plan inclin\'e. partie ii.
\newblock {\em Annales de la facult\'e de sciences de Toulouse},
  16(4):867--903, 2007.

\bibitem{Hillairet&Wittwer09}
M. Hillairet and P. Wittwer.
\newblock Existence of stationary solutions of the {N}avier-{S}tokes equations
  in two dimensions in the presence of a wall.
\newblock {\em J. Evol. Equ.}, 9(4):675--706, 2009.

\bibitem{Kavian93}
O.~Kavian.
\newblock {\em Introduction \`a la th\'eorie des points critiques et
  applications aux probl\`emes elliptiques}, volume~13 of {\em Math\'ematiques
  \& Applications (Berlin) [Mathematics \& Applications]}.
\newblock Springer-Verlag, Paris, 1993.

\bibitem{Leray33}
J.~Leray.
\newblock Etude de diverses \'equations int\'egrales non lin\'eaires et de
  quelques probl\`emes de l'hydrodynamique.
\newblock {\em J. Maths Pures Appl.}, 12:1--82, 1933.

\bibitem{Russo10}
A. Russo.
\newblock On the existence of d-solutions of the steady-state navier-stokes
  equations in plane exterior domains.
\newblock arXiv:1101.1243, 2010.

\bibitem{Sazonov03}
L.~I. Sazonov.
\newblock Justification of the asymptotic expansion of the solution of a
  two-dimensional flow problem for small {R}eynolds numbers.
\newblock {\em Izv. Ross. Akad. Nauk Ser. Mat.}, 67(5):125--154, 2003.

\bibitem{Serre87}
D.~Serre.
\newblock Chute libre d'un solide dans un fluide visqueux incompressible.
  {E}xistence.
\newblock {\em Japan J. Appl. Math.}, 4(1):99--110, 1987.

\bibitem{Weinberger73}
H.F. Weinberger.
\newblock On the steady fall of a body in a {N}avier-{S}tokes fluid.
\newblock In {\em Partial differential equations (Proc. Sympos. Pure Math.,
  Vol. XXIII, Univ. California, Berkeley, Calif., 1971)}, pages 421--439. Amer.
  Math. Soc., Providence, R. I., 1973.

\bibitem{Wittwer06}
P.~Wittwer.
\newblock Leading order down-stream asymptotics of stationary {N}avier-{S}tokes
  flows in three dimensions.
\newblock {\em Journal of Mathematical Fluid Mechanics}, 8:147--186, 2006.

\end{thebibliography}
\end{document}